\documentclass{article} 
\usepackage{iclr2020_conference,times}
\iclrfinalcopy

\usepackage[utf8]{inputenc} 
\usepackage[T1]{fontenc}    
\usepackage{hyperref}       
\usepackage{url}            
\usepackage{booktabs}       
\usepackage{amsfonts}       
\usepackage{nicefrac}       
\usepackage{microtype}      

\usepackage{nicefrac}       

\title{A Stochastic Derivative Free Optimization Method with Momentum}

\usepackage{amssymb, amsmath, amsthm, latexsym}
\usepackage{algorithm}
\usepackage{algpseudocode}
\usepackage{tabularx}
\usepackage{paralist}
\usepackage{mathtools}
\usepackage{tcolorbox}

\newcommand{\Exp}{\mathbf{E}}

\newcommand{\R}{\mathbb{R}}
\newcommand{\eqdef}{\stackrel{\text{def}}{=}}

\def\<#1,#2>{\left\langle #1,#2\right\rangle}

\usepackage{xcolor}

\def\epsilon{\varepsilon}

\newtheorem{theorem}{Theorem}[section]

\newtheorem{assumption}{Assumption}[section]

\newtheorem{lemma}{Lemma}[section]
\newtheorem{lem}{Lemma}[section]

\theoremstyle{plain}

\newcommand{\argmin}{\mathop{\arg\!\min}}

\usepackage[colorinlistoftodos,bordercolor=orange,backgroundcolor=orange!20,linecolor=orange,textsize=scriptsize]{todonotes}

\newcommand{\eduard}[1]{\todo[inline]{{\textbf{Eduard:} \emph{#1}}}}

\newcommand{\cD}{{\cal D}}

\newcommand{\EE}{\mathbf{E}}
\newcommand{\PP}{\mathbf{P}}

\usepackage{hyperref}

\author{Eduard Gorbunov\thanks{The research of Eduard Gorbunov was supported by RFBR, project number 18-31-20005 mol{\_}a{\_}ved}\\
   MIPT, Russia and IITP RAS, Russia and RANEPA, Russia \\
   \texttt{eduard.gorbunov@phystech.edu} \\
     \And
    Adel Bibi \\
    KAUST, Saudi Arabia \\
 \texttt{adel.bibi@kaust.edu.sa} \\
 \And
 Ozan Sener \\
 Intel Labs \\
 \texttt{ozan.sener@intel.com} \\
 \And
 El Houcine Bergou \\
 KAUST, Saudi Arabia and MaIAGE, INRA, France \\
 \texttt{elhoucine.bergou@inra.fr} \\
 \And
 Peter Richt\'arik \\
 KAUST, Saudi Arabia and MIPT, Russia\\
 \texttt{peter.richtarik@kaust.edu.sa} \\
}

\newcommand{\bluetext}{} 

\begin{document}

\maketitle

\begin{abstract}
    We consider the problem of unconstrained minimization of a smooth objective
function in $\mathbb{R}^d$
in setting where only function evaluations are possible. We propose and analyze stochastic zeroth-order method with heavy ball momentum. In particular, we propose, {\tt SMTP}, a momentum version of the stochastic three-point method ({\tt STP}) \cite{Bergou_2018}. We show new complexity results for non-convex, convex and strongly convex functions.  
We test our method on a collection of learning to continuous control tasks on several MuJoCo \cite{Todorov_2012} environments with varying
difficulty and compare against {\tt STP}, other state-of-the-art derivative-free optimization algorithms and against policy gradient methods. {\tt SMTP} significantly outperforms {\tt STP} and all other methods that we considered in our numerical experiments. Our second contribution is {\tt SMTP} with importance sampling which we call {\tt SMTP{\_}IS}. We provide convergence analysis of this method for non-convex, convex and strongly convex objectives.

\end{abstract}
\section{Introduction}
\label{introduction}

In this paper, we consider the following minimization problem
\begin{equation}\label{eq:problem}
	\min\limits_{x\in\R^d} f(x),
\end{equation}
where $f : \mathbb{R}^d \rightarrow \mathbb{R}$ is "smooth" but not necessarily a convex function in a Derivative-Free Optimization (DFO) setting where only function evaluations are possible. The function $f$ is bounded from below by $f(x^*)$ where  $x^*$ is a minimizer. Lastly and throughout the paper, we assume that $f$ is $L$-smooth.

\noindent \textbf{DFO.} In DFO setting \cite{Conn_2009, Kolda_2003}, the derivatives of the objective function $f$ are not accessible. That is they are either impractical to evaluate, noisy (function $f$ is noisy) \citep{chen2015stochastic} or they are simply not available at all. In standard applications of DFO, evaluations of $f$ are only accessible through simulations of black-box engine or software as in reinforcement learning and continuous control environments \cite{Todorov_2012}. This setting of optimization problems appears also  in applications from computational medicine \cite{Marsden_2008} and fluid dynamics \cite{Allaire_2001, Haslinger_2003, Mohammadi_2001} to localization \cite{Marsden_2004, Marsden_2007} and continuous control \cite{Mania_2018, Salimans_2017} to name a few.
 
The literature on DFO for solving (\ref{eq:problem}) is long and rich. The first approaches were based on deterministic direct search (DDS) and they span half a century of work \cite{Hooke_1961,1979positive,torczon1997convergence}. However, for DDS methods complexity bounds have only been established recently by the work of Vicente and coauthors \cite{vicente2013worst,dodangeh2016worst}. In particular, the work of Vicente \cite{vicente2013worst} showed the first complexity results on non-convex $f$ and the results were extended to better complexities when $f$ is convex  \cite{dodangeh2016worst}. However, there have been several variants of DDS, including randomized approaches \cite{Matyas_1965, Karmanov_1974a, Karmanov_1974b, Baba_1981, Dorea_1983, Sarma_1990}. Only very recently, complexity bounds have also been derived for randomized methods \cite{Diniz_2008,Stich_2011,ghadimi2013stochastic, ghadimi2016mini,Gratton_2015}. For instance, the work of \cite{Diniz_2008,Gratton_2015} imposes a decrease condition on whether to accept or reject a step of a set of random directions. Moreover, \cite{Nesterov_Spokoiny_2017} derived new complexity bounds when the random directions are normally distributed vectors for both smooth and non-smooth $f$. They proposed both accelerated and non-accelerated zero-order (ZO)  methods. Accelerated derivative-free methods in the case of inexact oracle information was proposed in \cite{dvurechensky2017randomized}. An extension of \cite{Nesterov_Spokoiny_2017} for non-Euclidean proximal setup was proposed by \cite{Dvurechensky_Gasnikov_Gorbunov_2018} for the smooth stochastic convex optimization with inexact oracle. {In \cite{stich2014convex,stich2014low} authors also consider acceleration of ZO methods and, in particular, develop the method called {\tt SARP}, proved that its convergence rate is not worse than for non-accelerated ZO methods and showed that in some cases it works even better.}
 
More recently and closely related to our work, \cite{Bergou_2018} proposed a new randomized direct search method called {\em Stochastic Three Points} (\texttt{STP}). At each iteration $k$ \texttt{STP} generates a random search direction $s_k$ according to a certain probability law and compares the objective function at three points: current iterate $x_k$, a point in the direction of $s_k$ and a point in the direction of $-s_k$ with a certain step size $\alpha_k$. The method then chooses the best of these three points as the new iterate:
 \[x_{k+1} = \arg\min \{f(x_k),  f(x_k + \alpha_k s_k), f(x_k - \alpha_k s_k)\}.\]
 
 The key properties of {\tt STP} are its simplicity, generality and practicality. Indeed, the update rule for {\tt STP} makes it extremely simple to implement, the proofs of convergence results for {\tt STP} {\bluetext are short and clear} and assumptions on random search directions cover a lot of strategies of choosing decent direction and even some of first-order methods fit the {\tt STP} scheme which makes it a very flexible in comparison with other zeroth-order methods (e.g.\ two-point evaluations methods like in \cite{Nesterov_Spokoiny_2017}, \cite{ghadimi2013stochastic}, \cite{ghadimi2016mini}, \cite{Dvurechensky_Gasnikov_Gorbunov_2018} that try to approximate directional derivatives along random direction at each iteration). Motivated by these properties of {\tt STP} we focus on further developing of this method.

\noindent \textbf{Momentum.}
Heavy ball momentum\footnote{We will refer to this as momentum.} is a special technique introduced by Polyak in 1964 \cite{polyak1964some} to get faster convergence to the optimum for the first-order methods. In the original paper, Polyak proved that his method converges \textit{locally} with  $O\left(\sqrt{\nicefrac{L}{\mu}}\log\nicefrac{1}{\varepsilon}\right)$ rate for twice continuously differentiable $\mu$-strongly convex and $L$-smooth functions. Despite the long history of this approach, there is still an open question whether heavy ball method converges to the optimum \textit{globally} with accelerated rate when the objective function is twice continuous differentiable, $L$-smooth and $\mu$-strongly convex. For this class of functions, only non-accelerated global convergence was proved \cite{ghadimi2015global} and for the special case of quadratic strongly convex and $L$-smooth functions Lessard et. al. \cite{lessard2016analysis} recently proved asymptotic accelerated global convergence. However, heavy ball method performs well in practice and, therefore, is widely used. One can find more detailed survey of the literature about heavy ball momentum in \cite{loizou2017momentum}.

\noindent \textbf{Importance Sampling.} Importance sampling has been celebrated and extensively studied in stochastic gradient based methods \cite{zhao2015stochastic} or in coordinate based methods \cite{richtarik2016optimal}. Only very recently, \cite{bibi19STPIS} proposed, {\tt STP{\_}IS}, the first DFO algorithm with importance sampling. In particular, under coordinate-wise smooth function, they show that sampling coordinate directions, can be generalized to arbitrary directions, with probabilities proportional to the function coordinate smoothness constants, improves the leading constant by the same factor typically gained in gradient based methods.

\textbf{Contributions.} Our contributions can be summarized into three folds.
\begin{itemize}
    \item \textbf{First ZO method with heavy ball momentum.} Motivated by practical effectiveness of first-order momentum heavy ball method, we introduce momentum into {\tt STP} method and propose new DFO algorithm with heavy ball momentum ({\tt SMTP}). We summarized the method in Algorithm \ref{alg:SMTP}, with theoretical guarantees for non-convex, convex and strongly convex functions under generic sampling directions $\cD$. We emphasize that the {\tt SMTP} with momentum is not a straightforward generalization of {\tt STP} and Polyak's method and requires insights from virtual iterates analysis from \cite{yang2016unified}.
    
    To the best of our knowledge it is the first analysis of derivative-free method with heavy ball momentum, i.e.\ we show that the same momentum trick that works for the first order method could be applied for zeroth-order methods as well. 
    \item \textbf{First ZO method with both heavy ball momentum and importance sampling.} In order to get more gain from momentum in the case when the sampling directions are coordinate directions and the objective function is coordinate-wise $L$-smooth (see Assumption \ref{as_sup:coord_wise_L_smoothness}), we consider importance sampling to the above method. In fact, we  propose the first zeroth-order momentum method with importance sampling ({\tt SMTP{\_}{IS}}) summarized in Algorithm \ref{alg:SMTP_IS} with theoretical guarantees for non-convex, convex and strongly convex functions. The details and proofs are left for Section \ref{sec:smtp_is_} and  Appendix \ref{sec:smtp_is}.
    \item \textbf{Practicality.} We conduct extensive experiments on continuous control tasks from the MuJoCo suite \cite{Todorov_2012} following recent success of DFO compared to model-free reinforcement learning \cite{Mania_2018,Salimans_2017}. We achieve with {\tt SMTP{\_}IS} the state-of-the-art results on across all tested environments on the continuous control outperforming DFO \cite{Mania_2018} and policy gradient methods \cite{schulman2015trust,Rajeswaran_2017}.
\end{itemize}



\begin{algorithm}[t]
   \caption{{\tt SMTP}: Stochastic Momentum Three Points}
   \label{alg:SMTP}
\begin{algorithmic}[1]
   \Require learning rates $\{\gamma^k\}_{k\geq 0}$, starting point $x^0 \in \R^d$, $\cD$~--- distribution on $\R^d$, $0\le \beta < 1$~--- momentum parameter
	\State Set $v^{-1} = 0$ and $z^0 = x^0$  
    \For{$k=0,1,\dotsc$}
       \State Sample $s^k \sim \cD$
       \State Let $v_+^k = \beta v^{k-1} + s^k$ and $v_-^k = \beta v^{k-1} - s^k$
       \State Let $x_+^{k+1} = x^{k} - \gamma^k v_+^k$ and $x_-^{k+1} = x^{k} - \gamma^k v_-^k$
       \State Let $z_+^{k+1} = x_+^{k+1} - \frac{\gamma^k\beta}{1-\beta}v_+^k$ and $z_-^{k+1} = x_-^{k+1} - \frac{\gamma^k\beta}{1-\beta}v_-^k$
       \State Set $z^{k+1} = \arg\min\left\{f(z^k), f(z_+^{k+1}), f(z_-^{k+1})\right\}$
       \State Set $x^{k+1} = \begin{cases}x_+^{k+1},\quad \text{if } z^{k+1} = z_+^{k+1}\\ x_-^{k+1},\quad \text{if } z^{k+1} = z_-^{k+1}\\x^{k},\quad \text{if } z^{k+1} = z^{k}\end{cases}$ and $v^{k+1} = \begin{cases}v_+^{k+1},\quad \text{if } z^{k+1} = z_+^{k+1}\\ v_-^{k+1},\quad \text{if } z^{k+1} = z_-^{k+1}\\v^{k},\quad \text{if } z^{k+1} = z^{k}\end{cases}$
   \EndFor
\end{algorithmic}
\end{algorithm}

\begin{table*}[t]
\centering
\small
\renewcommand{\arraystretch}{1.25}{
\begin{tabular}{c|c|c|c|c|c}
\toprule
Assumptions on $f$     & \begin{tabular}{@{}c@{}}{\tt SMTP}\\ Complexity \end{tabular} & Theorem & \begin{tabular}{@{}c@{}}Importance\\ Sampling \end{tabular} & \begin{tabular}{@{}c@{}}{\tt SMTP{\_}IS}\\Complexity \end{tabular} & Theorem  \\ 
  \midrule
None &  $\frac{2r_0L\gamma_\cD}{\mu_\cD^2 \epsilon^2}$ & \ref{thm:non_cvx} &$p_i = \frac{L_i}{\sum_{i=1}^d L_i}$ & $\frac{2 r_0 \textcolor{red}{ d\sum_{i=1}^d L_i}}{\epsilon^2}$ 
&   \ref{thm:non_cvx_is}\\

Convex, $R_0 < \infty$  & $\frac{1}{\varepsilon}\frac{L\gamma_\cD R_0^2}{\mu_\cD^2}\ln\left(\frac{2r_0}{\varepsilon}\right)$  & \ref{thm:cvx_constant_stepsize} & $p_i = \frac{L_i}{\sum_{i=1}^d L_i}$ & $\frac{R_0^2 \textcolor{red} { d\sum_{i=1}^d L_i}}{\epsilon} \ln\left(\frac{2r_0}{\varepsilon}\right)$  &   \ref{thm:cvx_constant_stepsize_is}\\

$\mu$-strongly convex & $ \frac{L}{\mu \mu_\cD^2}\ln\left(\frac{2r_0}{\varepsilon}\right)$ & \ref{thm:str_cvx_sol_free_stepsizes} & $p_i = \frac{L_i}{\sum_{i=1}^d L_i}$ &  $\frac{\textcolor{red}{\sum_{i=1}^d L_i}}{\mu}\ln\left(\frac{2r_0}{\varepsilon}\right) $& \ref{thm:str_cvx_sol_free_stepsizes_is} \\
\bottomrule
\end{tabular}
}
\vspace{-3pt}
\caption{Summary of the new derived complexity results of {\tt SMTP} and {\tt SMTP{\_}IS}. The complexities for {\tt SMTP} are under a generic sampling distribution $\cD$ satisfying Assumption \ref{ass:stp_general_assumption} while for {\tt SMTP{\_}IS} are under an arbitrary discrete sampling from a set of coordinate directions following \cite{bibi19STPIS} where we propose an importance sampling that improves the leading constant marked in red. Note that $r_0 = f(x_0) - f(x_*)$ and that all assumptions listed are in addition to Assumption \ref{ass:l_smoothness}. Complexity means number of iterations in order to guarantee $\EE\|\nabla f(\overline{z}^K)\|_{\cD} \le \varepsilon$ for the non-convex case, $\EE\left[f(z^K) - f(x^*)\right] \le \varepsilon$ for convex and strongly convex cases. $R_0 < \infty$ is the radius in $\|\cdot\|_{\cD}^*$-norm of a bounded level set where the exact definition is given in Assumption \ref{as:bounded_level_set}. We notice that for {\tt SMTP{\_}IS} $\|\cdot\|_{\cD} = \|\cdot\|_{1}$ and $\|\cdot\|_{\cD}^* = \|\cdot\|_\infty$ in non-convex and convex cases and $\|\cdot\|_{\cD} = \|\cdot\|_2$ in the strongly convex case.}
\label{tab:sumcompl}
\vspace{-15pt}
\end{table*}

We provide more detailed comparison of {\tt SMTP} and {\tt SMTP{\_}IS} in Section~\ref{sec_sup:comparison_with_smtp} of the Appendix.

\section{Notation and Definitions}\label{sec:notation}
We use $\|\cdot\|_{p}$ to define $\ell_p$-norm of the vector $x\in\R^d$: $\|x\|_p \eqdef \left(\sum_{i=1}^d |x_i|^p\right)^{\nicefrac{1}{p}}$ for $p \ge 1$ and $\|x\|_\infty \eqdef \max_{i\in[d]}|x_i|$ where $x_i$ is the $i$-th component of vector $x$, $[d] = \{1,2,\ldots,d\}$. Operator $\EE[\cdot]$ denotes mathematical expectation with respect to all randomness and $\EE_{s\sim\cD}[\cdot]$ denotes conditional expectation w.r.t. randomness coming from random vector $s$ which is sampled from probability distribution $\cD$ on $\R^d$. To denote standard inner product of two vectors $x,y\in\R^d$ we use $\langle x,y\rangle \eqdef \sum_{i=1}^d x_i y_i$, $e_i$ denotes $i$-th coordinate vector from standard basis in $\R^d$, i.e.\ $x = \sum_{i=1}^d x_i e_i$. {\bluetext We use $\|\cdot\|^{*}$ to define the conjugate norm for the norm $\|\cdot\|$: $\|x\|^{*} \eqdef \max\left\{\< a, x>\mid a\in \R^d, \|a\| \le 1\right\}$}.

As we mention in the introduction we assume throughout the paper\footnote{We will use thinner assumption in Section~\ref{sec:smtp_is_}.} that the objective function $f$ is $L$-smooth.
\begin{assumption}\label{ass:l_smoothness}($L$-smoothness)
	We say that $f$ is \textit{$L$-smooth} if
	\begin{equation}\label{eq:L_smoothness}
		\|\nabla f(x) - \nabla f(y)\|_2 \le L\|x-y\|_2 \quad \forall x,y\in\R^d.
	\end{equation}
\end{assumption} 
From this definition one can obtain
\begin{equation}\label{eq:quadratic_upper_bound}
	|f(y) - f(x) - \langle\nabla f(x), y-x \rangle | \le \frac{L}{2}\|y-x\|_2^2,\quad \forall x,y\in\R^d,
\end{equation}
and if additionally $f$ is convex, i.e.\ $f(y) \ge f(x) + \langle\nabla f(x), y-x \rangle$, we have
\begin{equation}\label{eq:gradient_upper_bound}
	\|\nabla f(x)\|_2^2 \le 2L(f(x) - f(x^*)), \quad \forall x\in\R^d.
\end{equation}

\section{ Stochastic Momentum Three Points ({\tt SMTP})} \label{sec:analysis}
Our analysis of {\tt SMTP} is based on the following key assumption.
\begin{assumption}\label{ass:stp_general_assumption}
	The probability distribution $\cD$ on $\R^d$ satisfies the following properties:
	\begin{enumerate}
	\item The quantity $\gamma_\cD \eqdef \EE_{s\sim\cD}\|s\|_2^2$ is {\bluetext finite}.
	\item There is a constant $\mu_\cD > 0$ {\bluetext for a norm $\|\cdot\|_\cD$ in} $\R^d$ such that for all $g\in\R^d$
	\begin{equation}\label{eq:inner_product_lower_bound}
		\EE_{s\sim\cD}|\langle g,s\rangle| \ge \mu_\cD\|g\|_{\cD}.
	\end{equation}
	\end{enumerate}
\end{assumption}

Some examples of distributions that meet above assumption are described in Lemma~3.4 from \cite{Bergou_2018}. For convenience we provide the statement of the lemma in the Appendix (see Lemma~\ref{lem_sup:aux_lemma}).

Recall that one possible view on {\tt STP} \cite{Bergou_2018} is as following. If we substitute gradient $\nabla f(x^k)$ in the update rule for the gradient descent $x^{k+1} = x^k - \gamma^k\nabla f(x^k)$ by $\pm s^k$ where $s^k$ is sampled from distribution $\cD$ satisfied Assumption~\ref{ass:stp_general_assumption} and then select $x^{k+1}$ as the best point in terms of functional value among $x^k, x^k - \gamma^ks^k, x^k + \gamma^ks^k$  we will get exactly {\tt STP} method. However, gradient descent is not the best algorithm to solve unconstrained smooth minimization problems and the natural idea is to try to perform the same substitution-trick with more efficient first-order methods than gradient descent.

We put our attention on Polyak's heavy ball method {\bluetext where the update rule} could be written in the following form:
\begin{equation}\label{eq:polyak_update}
    v^k = \beta v^{k-1} + \nabla f(x^k),\quad x^{k+1} = x^k - \gamma^k v^k.
\end{equation}
As in {\tt STP}, we substitute $\nabla f(x^k)$ by $\pm s^k$ and consider new sequences $\{v_+^k\}_{k\ge0}$ and $\{v_-^k\}_{k\ge 0}$ defined in the Algorithm~\ref{alg:SMTP}. However, it is not straightforward how to choose next $x^{k+1}$ and $v^k$ and the virtual iterates analysis \cite{yang2016unified} hints the update rule. We consider new iterates $z_+^{k+1} = x_+^{k+1} - \frac{\gamma^k\beta}{1-\beta}v_+^k$ and $z_-^{k+1} = x_-^{k+1} - \frac{\gamma^k\beta}{1-\beta}v_-^k$ and define $z^{k+1}$ as $\arg\min\left\{f(z^k), f(z_+^{k+1}), f(z_-^{k+1})\right\}$. Next we update $x^{k+1}$ and $v^k$ in order to have the same relationship between $z^{k+1},x^{k+1}$ and $v^k$ as between $z_+^{k+1},x_+^{k+1}$ and $v_+^k$ and $z_-^{k+1},x_-^{k+1}$ and $v_-^k$. Such scheme allows easily apply virtual iterates analysis and and generalize Key Lemma from \cite{Bergou_2018} which is the main tool in the analysis of {\tt STP}.

By definition of $z^{k+1}$, we get that the sequence $\{f(z^k)\}_{k\ge 0}$ is monotone:
\begin{equation}\label{eq:monotonicity}
	f(z^{k+1}) \le f(z^k) \qquad \forall k\ge 0.
\end{equation}

Now, we establish the key result which will be used to prove the main complexity results and remaining theorems in this section. 

\begin{lemma}\label{lem:key_lemma}
	Assume that $f$ is $L$-smooth and $\cD$ satisfies Assumption~\ref{ass:stp_general_assumption}. Then for the iterates of {\tt SMTP} the following inequalities hold:
	\begin{equation}\label{eq:key_lemma_without_expectation}
		f(z^{k+1}) \le f(z^k) - \frac{\gamma^k}{1-\beta}|\langle\nabla f(z^k),s^k\rangle| + \frac{L(\gamma^k)^2}{2(1-\beta)^2}\|s^k\|_2^2
	\end{equation}
	and
	\begin{equation}\label{eq:key_lemma}
		\EE_{s^k\sim\cD}\left[f(z^{k+1})\right] \le f(z^k) - \frac{\gamma^k\mu_\cD}{1-\beta}\|\nabla f(z^k)\|_\cD + \frac{L(\gamma^k)^2\gamma_\cD}{2(1-\beta)^2}.
	\end{equation}	 
\end{lemma}

\subsection{Non-Convex Case}\label{sec:non_cvx}

In this section, we show our complexity results for Algorithm \ref{alg:SMTP} in the case when f is allowed to be non-convex. In particular, we show that {\tt SMTP} in Algorithm \ref{alg:SMTP} guarantees complexity bounds with the same order as classical bounds, i.e. $1/\sqrt{K}$ where $K$ is the number of iterations, in the literature. We notice that query complexity (i.e.\ number of oracle calls) of {\tt SMTP} coincides with its iteration complexity up to numerical constant factor. For clarity and completeness, proofs are left for the appendix.


\begin{theorem}\label{thm:non_cvx}
	Let Assumptions \ref{ass:l_smoothness} and \ref{ass:stp_general_assumption} be satisfied. Let {\tt SMTP} with $\gamma^k \equiv \gamma > 0$ produce points $\{z^0,z^1,\ldots, z^{K-1}\}$ and $\overline{z}^K$ is chosen uniformly at random among them. Then
	\begin{equation}\label{eq:non_cvx}
	\EE\left[\|\nabla f(\overline{z}^K)\|_\cD\right] \le \frac{(1-\beta)(f(x^0) - f(x^*))}{K\gamma\mu_\cD} + \frac{L\gamma\gamma_\cD}{2\mu_\cD(1-\beta)}.
	\end{equation}
	Moreover, if we choose $\gamma = \frac{\gamma_0}{\sqrt{K}}$ the complexity \eqref{eq:non_cvx} reduces to
		\begin{equation}\label{eq:cor_non_cvx}
		\EE\left[\|\nabla f(\overline{z}^K)\|_\cD\right] \le \frac{1}{\sqrt{K}}\left(\frac{(1-\beta)(f(z^0) - f(x^*))}{\gamma_0\mu_\cD} + \frac{L\gamma_0\gamma_\cD}{2\mu_\cD(1-\beta)}\right).
	\end{equation}
    Then $\gamma_0 = \sqrt{\frac{2(1-\beta)^2(f(x^0) - f(x^*))}{L\gamma_\cD}}$ minimizes the right-hand side of {\bluetext \eqref{eq:cor_non_cvx}} and for this choice we have
	\begin{equation}\label{eq_sup:cor_non_cvx_optimal_gamma}
		\EE\left[\|\nabla f(\overline{z}^K)\|_\cD\right] \le \frac{\sqrt{2\left(f(x^0)-f(x^*)\right)L\gamma_\cD}}{\mu_\cD\sqrt{K}}.
	\end{equation}
\end{theorem}
In other words, the above theorem states that {\tt SMTP} converges no worse than {\tt STP} for non-convex problems to the stationary point.  In the next sections we also show that theoretical convergence guarantees for {\tt SMTP} are not worse than for {\tt STP} for convex and strongly convex problems. However, in practice {\tt SMTP} significantly outperforms {\tt STP}. So, the relationship between {\tt SMTP} and {\tt STP} correlates with the known in the literature relationship between Polyak's heavy ball method and gradient descent.


\subsection{Convex Case}\label{sec:cvx}
In this section, we present our complexity results for Algorithm \ref{alg:SMTP} when $f$ is convex. In particular, we show that this method guarantees complexity bounds with the same order as classical bounds, i.e. $1/K$, in the literature. We will need the following additional assumption in the sequel.

\begin{assumption}\label{as:bounded_level_set}
	We assume that $f$ is convex, has a minimizer $x^*$ and has bounded level set at $x^0$:
	\begin{equation}\label{eq:bounded_level_set}
		R_0 \eqdef \max\left\{\|x-x^*\|_\cD^* \mid f(x) \le f(x^0)\right\} < +\infty,
	\end{equation}
	where $\|\xi\|_\cD^* \eqdef \max\left\{\langle\xi,x\rangle\mid \|x\|_\cD \le 1\right\}$ defines the dual norm to $\|\cdot\|_\cD$.
\end{assumption} 
From the above assumption and Cauchy-Schwartz inequality we get the following implication:
\[
	f(x) \le f(x_0) \Longrightarrow f(x) - f(x_*) \le \langle\nabla f(x), x-x^* \rangle \le \|\nabla f(x)\|_\cD\|x-x^*\|_\cD^* \le R_0\|\nabla f(x)\|_\cD,
\] 
which implies
\begin{equation}\label{eq:gradient_lower_bound}
	\|\nabla f(x)\|_\cD \ge \frac{f(x)-f(x^*)}{R_0}\qquad \forall x: f(x) \le f(x_0).
\end{equation}
\begin{theorem}[Constant stepsize]\label{thm:cvx_constant_stepsize}
	 Let Assumptions \ref{ass:l_smoothness}, \ref{ass:stp_general_assumption} and \ref{as:bounded_level_set} be satisfied. If we set $\gamma^k\equiv \gamma  < \frac{(1-\beta)R_0}{\mu_\cD}$, then for the iterates of {\tt SMTP} method the following inequality holds:
	\begin{equation}\label{eq:cvx_constant_stepsize}
	 \EE\left[f(z^k)-f(x^*)\right] \le \left(1 - \frac{\gamma\mu_\cD}{(1-\beta)R_0}\right)^k\left(f(x^0)-f(x^*)\right) + \frac{L\gamma\gamma_\cD R_0}{2(1-\beta)\mu_\cD}.
	\end{equation}
	If we choose $\gamma = \frac{\varepsilon(1-\beta)\mu_\cD}{L\gamma_\cD R_0}$ for some $0<\varepsilon\le \frac{L\gamma_\cD R_0^2}{\mu_\cD^2}$ and run {\tt SMTP} for $k = K$ iterations where
	\begin{equation}\label{eq:cvx_constant_stepsize_number_of_iterations}
		K = \frac{1}{\varepsilon}\frac{L\gamma_\cD R_0^2}{\mu_\cD^2}\ln\left(\frac{2(f(x^0) - f(x^*))}{\varepsilon}\right),
	\end{equation}
	then we will get $\EE\left[f(z^K)\right] - f(x^*) \le \varepsilon$.
\end{theorem}




In order to get rid of factor $\ln\frac{2(f(x^0)-f(x^*))}{\varepsilon}$ in the complexity we consider decreasing stepsizes.
\begin{theorem}[Decreasing stepsizes]\label{thm:cvx_decreasing_stepsizes}
    Let Assumptions \ref{ass:l_smoothness}, \ref{ass:stp_general_assumption} and \ref{as:bounded_level_set} be satisfied. If we set $\gamma^k = \frac{2}{\alpha k + \theta}$, where $\alpha = \frac{\mu_\cD}{(1-\beta)R_0}$ and $\theta \ge \frac{2}{\alpha}$, then for the iterates of {\tt SMTP} method the following inequality holds:
   	\begin{equation}\label{eq:cvx_decreasing_stepsizes}
   		\EE\left[f(z^k)\right] - f(x^*) \le \frac{1}{\eta k+1}\max\left\{ f(x^0) - f(x^*), \frac{2L\gamma_\cD}{\alpha\theta(1-\beta)^2}\right\},	
   	\end{equation}
    where $\eta\eqdef \frac{\alpha}{\theta}$. Then, if we choose $\gamma^k = \frac{2\alpha}{\alpha^2k+2}$ where $\alpha = \frac{\mu_\cD}{(1-\beta)R_0}$ and run {\tt SMTP} for $k = K$ iterations where
	\begin{equation}\label{eq:cvx_decreasing_stepsizes_number_of_iterations}
		K = \frac{1}{\varepsilon}\cdot\frac{2R_0^2}{\mu_\cD^2}\max\left\{(1-\beta)^2(f(x^0)-f(x^*)), L\gamma_\cD\right\} - \frac{2(1-\beta)^2R_0^2}{\mu_\cD^2},\qquad \varepsilon > 0,
	\end{equation}
	 we get $\EE\left[f(z^K)\right] - f(x^*) \le \varepsilon$.
\end{theorem}

We notice that if we choose $\beta$ sufficiently close to $1$, we will obtain from the formula \eqref{eq:cvx_decreasing_stepsizes_number_of_iterations} that $K \approx \frac{2R_0^2L\gamma_{\cD}}{\varepsilon\mu_{\cD}^2}$. 


\subsection{Strongly Convex Case}\label{sec:str_cvx}
In this section we present our complexity results for Algorithm \ref{alg:SMTP} when $f$ is $\mu$-strongly convex. 

\begin{assumption}\label{as:strong_convexity}
	We assume that $f$ is $\mu$-strongly convex with respect to the norm {\bluetext $\|\cdot\|_{\cD}^*$}:
	\begin{equation}\label{eq:strong_convexity}
	   f(y) \ge f(x) + \<\nabla f(x), y-x > + \frac{\mu}{2}{\bluetext \left(\|y-x\|_{\cD}^*\right)}^2,\quad \forall x,y\in\R^d.
	\end{equation}
\end{assumption}
It is well known that strong convexity implies
\begin{equation}\label{eq:consequence_of_strong_convexity}
	\|\nabla f(x)\|_\cD^2 \ge 2\mu\left(f(x) - f(x^*)\right).
\end{equation}


\begin{theorem}[Solution-dependent stepsizes]\label{thm:str_cvx_sol_dep_stepsizes}
Let Assumptions \ref{ass:l_smoothness}, \ref{ass:stp_general_assumption} and \ref{as:strong_convexity} be satisfied. If we set $\gamma^k = \frac{(1-\beta)\theta_k\mu_\cD}{L}\sqrt{2\mu(f(z^k)-f(x^*))}$ for some $\theta_k\in(0,2)$ such that $\theta = \inf\limits_{k\ge 0}\{2\theta_k - \gamma_\cD\theta_k^2\} \in \left(0,\nicefrac{L}{(\mu_\cD^2\mu)}\right)$, then for the iterates of {\tt SMTP}, the following inequality holds:
	\begin{equation}\label{eq:str_cvx_sol_dep_stepsizes}
		\EE\left[f(z^k)\right] - f(x^*) \le \left(1 - \frac{\theta\mu_\cD^2\mu}{L}\right)^{k}\left(f(x^0) - f(x^*)\right).
	\end{equation}
	Then, If we run {\tt SMTP} for $k = K$ iterations where
	\begin{equation}\label{eq:str_cvx_sol_dep_stepsizes_number_of_iterations}
		K = \frac{\kappa}{\theta\mu_\cD^2}\ln\left(\frac{f(x^0)-f(x^*)}{\varepsilon}\right),\qquad \varepsilon > 0,
	\end{equation}
	 where $\kappa\eqdef \frac{L}{\mu}$ is the condition number of the objective, we will get $\EE\left[f(z^K)\right] - f(x^*) \le \varepsilon$.
\end{theorem}


Note that the previous result uses stepsizes that depends on the optimal solution $f(x^*)$ which is often not known in practice. The next theorem removes this drawback without spoiling the convergence rate. However, we need an additional assumption on the distribution $\cD$ and one extra function evaluation.

\begin{assumption}\label{as:distr_unit_vectors}
	We assume that for all $s\sim\cD$ we have $\|s\|_2 = 1$.
\end{assumption}

\begin{theorem}[Solution-free stepsizes]\label{thm:str_cvx_sol_free_stepsizes}
Let Assumptions \ref{ass:l_smoothness}, \ref{ass:stp_general_assumption}, \ref{as:strong_convexity} and \ref{as:distr_unit_vectors} be satisfied. If additionally we compute $f(z^k + ts^k)$, set $\gamma^k = \nicefrac{(1-\beta)|f(z^k + ts^k) - f(z^k)|}{(L t)}$ for $t>0$ and assume that $\cD$ is such that $\mu_\cD^2 \le \nicefrac{L}{\mu}$, then for the iterates of {\tt SMTP} the following inequality holds:
	\begin{equation}\label{eq:str_cvx_sol_free_stepsizes}
		\EE\left[f(z^k)\right] - f(x^*) \le \left(1 - \frac{\mu_\cD^2\mu}{L}\right)^k\left(f(x^0) - f(x^*)\right) + \frac{L^2t^2}{8\mu_\cD^2\mu}.
	\end{equation}
	Moreover, for any $\varepsilon > 0$ if we set $t$ such that
	\begin{equation}\label{eq:str_cvx_sol_free_stepsizes_t_bound}
		0 < t \le \sqrt{\frac{4\varepsilon\mu_\cD^2\mu}{L^2}},
	\end{equation}		
	 and run {\tt SMTP} for $k = K$ iterations where
	\begin{equation}\label{eq:str_cvx_sol_free_stepsizes_number_of_iterations}
		K = \frac{\kappa}{\mu_\cD^2}\ln\left(\frac{2(f(x^0)-f(x^*))}{\varepsilon}\right),
	\end{equation}
	 where $\kappa\eqdef \frac{L}{\mu}$ is the condition number of $f$, we will have $\EE\left[f(z^K)\right] - f(x^*) \le \varepsilon$.
\end{theorem}


\section{ Stochastic Momentum Three Points with Importance Sampling ({\tt SMTP{\_}IS})}\label{sec:smtp_is_}

In this section we consider another assumption, in a similar spirit to \cite{bibi19STPIS}, on the objective.
\begin{assumption}[Coordinate-wise $L$-smoothness]\label{as_sup:coord_wise_L_smoothness}
	We assume that the objective $f$ has coordinate-wise Lipschitz gradient, with Lipschitz constants $L_1,\ldots,L_d > 0$, i.e.\ 
	\begin{equation}\label{eq_sup:coord_wise_L_smoothness}
		f(x+he_i) \le f(x) + \nabla_if(x)h + \frac{L_i}{2}h^2,\qquad \forall x\in\R^d, h\in\R, 
	\end{equation}
	where $\nabla_i f(x)$ is $i$-th partial derivative of $f$ at the point $x$.
\end{assumption}

For this kind of problems we modify {\tt SMTP} and present {\tt STMP{\_}IS} method in  Algorithm~\ref{alg:SMTP_IS}. In general, the idea behind methods with importance sampling and, in particular, behind {\tt SMTP{\_}IS} is to adjust probabilities of sampling in such a way that gives better convergence guarantees. In the case when $f$ satisfies coordinate-wise $L$-smoothness and Lipschitz constants $L_i$ are known it is natural to sample direction $s^k = e_i$ with probability depending on $L_i$ (e.g.\ proportional to $L_i$). One can find more detailed discussion of the importance sampling in \cite{zhao2015stochastic} and \cite{richtarik2016optimal}.

\begin{algorithm}[t]
   \caption{{\tt SMTP{\_}IS}: Stochastic Momentum Three Points with Importance Sampling}
   \label{alg:SMTP_IS}
\begin{algorithmic}[1]
   \Require stepsize parameters $w_1,\ldots,w_n > 0$, probabilities $p_1,\ldots, p_n > 0$ summing to $1$, starting point $x^0 \in \R^d$, $0\le \beta < 1$~--- momentum parameter
	\State Set $v^{-1} = 0$ and $z^0 = x^0$  
    \For{$k=0,1,\dotsc$}
       \State Select $i_k = i$ with probability $p_i > 0$
	   \State Choose stepsize $\gamma_i^k$ proportional to $\frac{1}{w_{i_k}}$        
       \State Let $v_+^k = \beta v^{k-1} + e_{i_k}$ and $v_-^k = \beta v^{k-1} - e_{i_k}$
       \State Let $x_+^{k+1} = x^{k} - \gamma_i^k v_+^k$ and $x_-^{k+1} = x^{k} - \gamma_i^k v_-^k$
       \State Let $z_+^{k+1} = x_+^{k+1} - \frac{\gamma_i^k\beta}{1-\beta}v_+^k$ and $z_-^{k+1} = x_-^{k+1} - \frac{\gamma_i^k\beta}{1-\beta}v_-^k$
       \State Set $z^{k+1} = \arg\min\left\{f(z^k), f(z_+^{k+1}), f(z_-^{k+1})\right\}$
       \State Set $x^{k+1} = \begin{cases}x_+^{k+1},\quad \text{if } z^{k+1} = z_+^{k+1}\\ x_-^{k+1},\quad \text{if } z^{k+1} = z_-^{k+1}\\x^{k},\quad \text{if } z^{k+1} = z^{k}\end{cases}$ and $v^{k+1} = \begin{cases}v_+^{k+1},\quad \text{if } z^{k+1} = z_+^{k+1}\\ v_-^{k+1},\quad \text{if } z^{k+1} = z_-^{k+1}\\v^{k},\quad \text{if } z^{k+1} = z^{k}\end{cases}$
   \EndFor
\end{algorithmic}
\end{algorithm}

Now, we establish the key result which will be used to prove the main complexity results of {\tt STMP{\_}IS}. 
\begin{lem}\label{lem_sup:key_lemma_is}
	Assume that $f$ satisfies Assumption~\ref{as_sup:coord_wise_L_smoothness}. Then for the iterates of {\tt SMTP{\_}IS} the following inequalities hold:
	\begin{equation}\label{eq_sup:key_lemma_without_expectation_is}
		f(z^{k+1}) \le f(z^k) - \frac{\gamma_i^k}{1-\beta}|\nabla_{i_k} f(z^k)| + \frac{L_{i_k}(\gamma_i^k)^2}{2(1-\beta)^2}
	\end{equation}
	and
	\begin{equation}\label{eq_sup:key_lemma_is}
		\EE_{s^k\sim\cD}\left[f(z^{k+1})\right] \le f(z^k) - \frac{1}{1-\beta}\EE\left[\gamma_i^k|\nabla_{i_k} f(z^k)|\mid z^k\right] + \frac{1}{2(1-\beta)^2}\EE\left[L_{i_k}(\gamma_i^k)^2\mid z^k\right].
	\end{equation}	 
\end{lem}

Due to the page limitation, we provide the complexity results of {\tt SMTP{\_}IS} in the Appendix.  
\section{Experiments}
\label{experiments}

\textbf{Experimental Setup.}
We conduct extensive experiments\footnote{The code will be made available online upon acceptance of this work.}
on challenging non-convex problems on the continuous control task from the MuJoCO suit \cite{Todorov_2012}. In particular, we address the problem of model-free control of a dynamical system. Policy gradient methods for model-free reinforcement learning algorithms provide an off-the-shelf model-free approach to learn how to control a dynamical system and are often benchmarked in a simulator. We compare our proposed momentum stochastic three points method {\tt SMTP} and the momentum with importance sampling version {\tt SMTP{\_}IS} against state-of-art DFO based methods as {\tt STP{\_}IS}  \cite{bibi19STPIS} and ARS \cite{Mania_2018}. Moreover, we also compare against classical policy gradient methods as TRPO \cite{schulman2015trust} and NG \cite{Rajeswaran_2017}. We conduct experiments on several environments with varying difficulty \texttt{Swimmer-v1}, \texttt{Hopper-v1}, \texttt{HalfCheetah-v1}, \texttt{Ant-v1}, and \texttt{Humanoid-v1}. 

Note that due to the stochastic nature of problem where $f$ is stochastic, we use the mean of the function values of $f(x^k),f(x^k_+)$ and $f(x^k_-)$, see Algorithm \ref{alg:SMTP}, over K observations. Similar to the  work in  \cite{bibi19STPIS}, we use $K=2$ for \texttt{Swimmer-v1}, $K = 4$ for both \texttt{Hopper-v1} and \texttt{HalfCheetah-v1}, $K=40$ for \texttt{Ant-v1} and  \texttt{Humanoid-v1}. Similar to \cite{bibi19STPIS}, these values were chosen based on the validation performance over the grid that is $K \in \{1,2,4,8,16\}$ for the smaller dimensional problems \texttt{Swimmer-v1}, \texttt{Hopper-v1}, \texttt{HalfCheetah-v1} and $K \in \{20,40,80, 120\}$ for larger dimensional problems \texttt{Ant-v1}, and \texttt{Humanoid-v1}. As for the momentum term, for {\tt SMTP} we set $\beta = 0.5$. For {\tt SMTP{\_}IS}, as the smoothness constants are not available for continuous control, we use the coordinate smoothness constants of a $\theta$ parameterized smooth function $\hat{f}_\theta$ (multi-layer perceptron) that estimates $f$. In particular, consider running any DFO for n steps; with the queried sampled $\{x_i,f(x_i)\}_{i=1}^{n}$, we estimate $f$ by solving \mbox{$\theta_{n+1} = \argmin_{\theta} \sum_i (f(x_i) - \hat{f}(x_i;\theta))^2$}. See \cite{bibi19STPIS} for further implementation details as we follow the same experimental procedure. In contrast to {\tt STP{\_}IS}, our method ({\tt SMTP}) does not required sampling from directions in the canonical basis; hence, we use directions from standard Normal distribution in each iteration. For {\tt SMTP{\_}IS}, we follow a similar procedure as \cite{bibi19STPIS} and sample from columns of a random matrix $B$.

Similar to the standard practice, we perform all experiments with 5 different initialization and measure the average reward, in continuous control we are maximizing the reward function $f$, and best and worst run per iteration. We compare algorithms in terms of reward vs. sample complexity.

\textbf{Comparison Against {\tt STP}}. Our method improves sample complexity of {\tt STP} and {\tt STP{\_}IS} significantly.  Especially for high dimensional problems like {\tt Ant-v1} and {\tt Humanoid-v1}, sample efficiency of {\tt SMTP} is at least as twice as the {\tt STP}. Moreover, {\tt SMTP{\_}IS} helps in some experiments by improving over {\tt SMTP}. However, this is not consistent in all environments. We believe this is largely due to the fact that {\tt SMTP{\_}IS} can only handle sampling from canonical basis similar to {\tt STP{\_}IS}.

\textbf{Comparison Against State-of-The-Art}. We compare our method with state-of-the-art DFO and policy gradient algorithms. For the environments, {\tt Swimmer-v1, Hopper-v1, HalfCheetah-v1} and {\tt Ant-v1}, our method outperforms the state-of-the-art results. Whereas for {\tt Humanoid-v1}, our methods results in a comparable sample complexity. 

\begin{figure}
    \centering
    \includegraphics[width=\columnwidth]{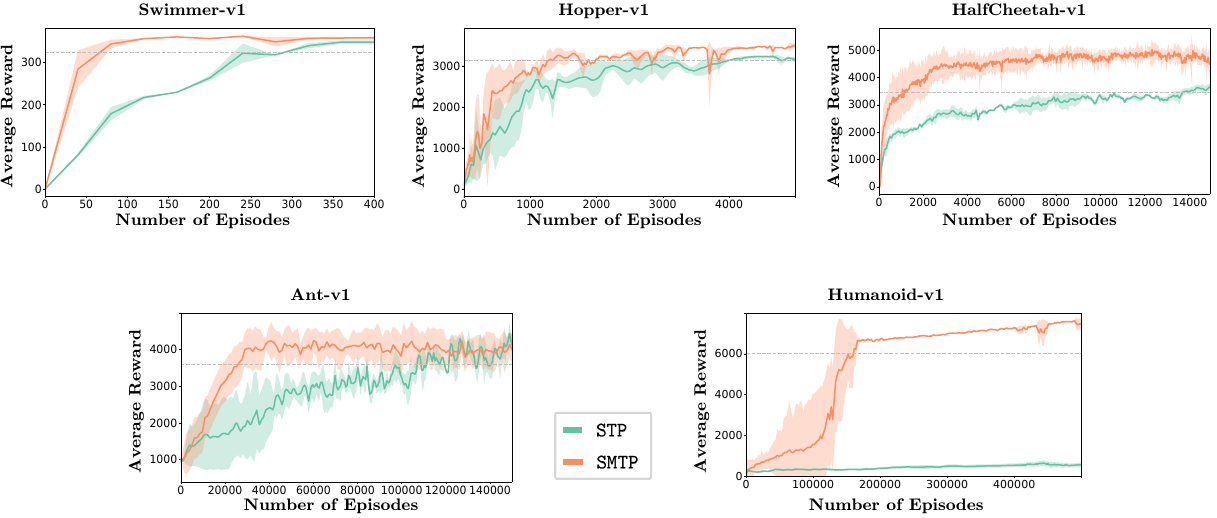}
    \caption{{\tt SMTP} {\bluetext is far superior} to {\tt STP} on all 5 different MuJoCo tasks particularly on the high dimensional \texttt{Humanoid-v1} problem. The horizontal dashed lines are the thresholds used in Table \ref{tab:rl} to demonstrate complexity of each method.}
    \label{fig:my_label}
\end{figure}

\begin{table*}[t]
\vspace{-2mm}
\caption{For each MuJoCo task, we report the average number of episodes required to achieve a predefined reward threshold. Results for our method is averaged over five random seeds, the rest is copied from \citep{Mania_2018} (N/A means the method failed to reach the threshold. UNK means the results is unknown since they are not reported in the literature.)}
  \centering
 \resizebox{\textwidth}{!}{
\begin{tabular}{rccccccccc}
  \toprule
    & Threshold	& \texttt{STP} & \texttt{STP}$_{\text{\texttt{IS}}}$ & \texttt{SMTP} & \texttt{SMTP}$_{\text{\texttt{IS}}}$ & ARS(V1-t) & ARS(V2-t) &	NG-lin	& TRPO-nn \\
    \midrule
    \texttt{Swimmer-v1} & 325 & 320 & 110 & 80 &    100 & 100 & 427 & 1450 & N/A  \\
    \texttt{Hopper-v1} & 3120 & 3970 & 2400  &1264  & 1408   & 51840  & 1973 & 13920 & 10000 \\
    \texttt{HalfCheetah-v1} & 3430 & 13760 & 4420 & 1872 & 1624  & 8106 & 1707 & 11250 & 4250  \\
    \texttt{Ant-v1} & 3580 & 107220 & 43860 & 19890 & 14420 & 58133 & 20800  & 39240 & 73500 \\
    \texttt{Humanoid-v1} & 6000 & N/A & 530200 & 161230 &  207160 & N/A & 142600  & 130000  & UNK \\
      \bottomrule
\end{tabular}}
\label{tab:rl}
\vspace{-2mm}
\vspace{-0.25cm}
\vspace{-2pt}
\end{table*}
\section{Conclusion}
\label{conclusion}
We have proposed, {\tt SMTP}, the first heavy ball momentum DFO based algorithm with convergence rates for non-convex, convex and strongly convex functions under generic sampling direction. We specialize the sampling to the set of coordinate bases and further improve rates by proposing a momentum and importance sampling version {\tt SMPT{\_}IS} with new convergence rates for non-convex, convex and strongly convex functions too. We conduct large number of experiments on the task of controlling dynamical systems. {\bluetext We outperform two different policy gradient methods and achieve comparable or better performance to the best DFO algorithm (ARS) on the respective environments.}

{\small
\bibliography{iclr2020_conference}
\bibliographystyle{iclr2020_conference}
}

\newpage
\appendix
\part*{\Large A Stochastic Derivative Free Optimization Method with Momentum\\ \phantom{xx}
 \Large (Supplementary Material)}

\section{Preliminaries}\label{sec_sup:preliminaries}

We first list the main assumptions.
\begin{assumption}\label{ass_sup:l_smoothness}($L$-smoothness)
	We say that $f$ is \textit{$L$-smooth} if:
	\begin{equation}\label{eq_sup:L_smoothness}
		\|\nabla f(x) - \nabla f(y)\|_2 \le L\|x-y\|_2 \quad \forall x,y\in\R^d.
	\end{equation}
\end{assumption} 

\begin{assumption}\label{ass_sup:stp_general_assumption}
	The probability distribution $\cD$ on $\R^d$ satisfies the following properties:
	\begin{enumerate}
	\item The quantity $\gamma_\cD \eqdef \EE_{s\sim\cD}\|s\|_2^2$ is positive and finite.
	\item There is a constant $\mu_\cD > 0$ and norm $\|\cdot\|_\cD$ on $\R^d$ such that for all $g\in\R^d$
	\begin{equation}\label{eq_sup:inner_product_lower_bound}
		\EE_{s\sim\cD}|\langle g,s\rangle| \ge \mu_\cD\|g\|_{\cD}.
	\end{equation}
	\end{enumerate}
\end{assumption}

We establish the key lemma which will be used to prove the theorems stated in the paper.
\begin{lem}\label{lem_sup:key_lemma}
	Assume that $f$ is $L$-smooth and $\cD$ satisfies Assumption~\ref{ass_sup:stp_general_assumption}. Then for the iterates of {\tt SMTP} the following inequalities hold:
	\begin{equation}\label{eq_sup:key_lemma_without_expectation}
		f(z^{k+1}) \le f(z^k) - \frac{\gamma^k}{1-\beta}|\langle\nabla f(z^k),s^k\rangle| + \frac{L(\gamma^k)^2}{2(1-\beta)^2}\|s^k\|_2^2
	\end{equation}
	and
	\begin{equation}\label{eq_sup:key_lemma}
		\EE_{s^k\sim\cD}\left[f(z^{k+1})\right] \le f(z^k) - \frac{\gamma^k\mu_\cD}{1-\beta}\|\nabla f(z^k)\|_\cD + \frac{L(\gamma^k)^2\gamma_\cD}{2(1-\beta)^2}.
	\end{equation}	 
\end{lem}
\begin{proof}
	By induction one can show that
	\begin{equation}\label{eq:z_recursion}
		z^{k} = x^k - \frac{\gamma^k\beta}{1-\beta}v^{k-1}.
	\end{equation}
	That is, for $k = 0$ this recurrence holds and update rules for $z^{k}, x^k$ and $v^{k-1}$ do not brake it.
	From this we get
	\begin{eqnarray*}
		z_+^{k+1} &=& x_+^{k+1} - \frac{\gamma^k\beta}{1-\beta}v_+^k = x^k - \gamma^k v_+^k - \frac{\gamma^k\beta}{1-\beta}v_+^k\\
		&=& x^k - \frac{\gamma^k}{1-\beta}v_+^k = x^k - \frac{\gamma^k\beta}{1-\beta}v^{k-1} - \frac{\gamma^k}{1-\beta}s^k\\
		&\overset{\eqref{eq:z_recursion}}{=}& z^k - \frac{\gamma^k}{1-\beta}s^k.
	\end{eqnarray*}
	Similarly,
	\begin{eqnarray*}
		z_-^{k+1} &=& x_-^{k+1} - \frac{\gamma^k\beta}{1-\beta}v_-^k = x^k - \gamma^k v_-^k - \frac{\gamma^k\beta}{1-\beta}v_-^k\\
		&=& x^k - \frac{\gamma^k}{1-\beta}v_-^k = x^k - \frac{\gamma^k\beta}{1-\beta}v^{k-1} + \frac{\gamma^k}{1-\beta}s^k\\
		&\overset{\eqref{eq:z_recursion}}{=}& z^k + \frac{\gamma^k}{1-\beta}s^k.
	\end{eqnarray*}
	It implies that
	\begin{eqnarray*}
		f(z_+^{k+1}) &\overset{\eqref{eq:quadratic_upper_bound}}{\le}& f(z^k) + \langle\nabla f(z^k), z_+^{k+1} - z_k \rangle + \frac{L}{2}\|z_+^{k+1}-z^k\|_2^2\\
		&=& f(z^k) - \frac{\gamma^k}{1-\beta}\langle\nabla f(z^k),s^k\rangle + \frac{L(\gamma^k)^2}{2(1-\beta)^2}\|s^k\|_2^2	
	\end{eqnarray*}
	and
	\begin{eqnarray*}
		f(z_-^{k+1}) &\le& f(z^k) + \frac{\gamma^k}{1-\beta}\langle\nabla f(z^k),s^k\rangle + \frac{L(\gamma^k)^2}{2(1-\beta)^2}\|s^k\|_2^2.	
	\end{eqnarray*}
	Unifying these two inequalities we get
	\begin{eqnarray*}
		f(z^{k+1}) &\le& \min\{f(z_+^{k+1}),f(z_-^{k+1})\} = f(z^k) - \frac{\gamma^k}{1-\beta}|\langle\nabla f(z^k),s^k\rangle| + \frac{L(\gamma^k)^2}{2(1-\beta)^2}\|s^k\|_2^2,
	\end{eqnarray*}
	which proves \eqref{eq_sup:key_lemma_without_expectation}.
	Finally, taking the expectation $\EE_{s^k\sim\cD}$ of both sides of the previous inequality and invoking Assumption~\ref{ass_sup:stp_general_assumption}, we obtain
	\begin{equation*}
		\EE_{s^k\sim\cD}\left[f(z^{k+1})\right] \le  f(z^k) - \frac{\gamma^k\mu_\cD}{1-\beta}\|\nabla f(z^k)\|_\cD + \frac{L(\gamma^k)^2\gamma_\cD}{2(1-\beta)^2}.
	\end{equation*}
\end{proof}

\section{Non-Convex Case}\label{sec_sup:non_cvx}

\begin{theorem}\label{thm_sup:non_cvx}
		Let Assumptions \ref{ass_sup:l_smoothness} and \ref{ass_sup:stp_general_assumption} be satisfied. Let {\tt SMTP} with $\gamma^k \equiv \gamma > 0$ produce points $\{z^0,z^1,\ldots, z^{K-1}\}$ and $\overline{z}^K$ is chosen uniformly at random among them. Then
	\begin{equation}\label{eq_sup:non_cvx}
	\EE\left[\|\nabla f(\overline{z}^K)\|_\cD\right] \le \frac{(1-\beta)(f(x^0) - f(x^*))}{K\gamma\mu_\cD} + \frac{L\gamma\gamma_\cD}{2\mu_\cD(1-\beta)}.
	\end{equation}
	Moreover, if we choose $\gamma = \frac{\gamma_0}{\sqrt{K}}$ the complexity \eqref{eq_sup:non_cvx} reduces to
		\begin{equation}\label{eq_sup:cor_non_cvx}
		\EE\left[\|\nabla f(\overline{z}^K)\|_\cD\right] \le \frac{1}{\sqrt{K}}\left(\frac{(1-\beta)(f(z^0) - f(x^*))}{\gamma_0\mu_\cD} + \frac{L\gamma_0\gamma_\cD}{2\mu_\cD(1-\beta)}\right).
	\end{equation}
    Then $\gamma_0 = \sqrt{\frac{2(1-\beta)^2(f(x^0) - f(x^*))}{L\gamma_\cD}}$ minimizes the right-hand side of \eqref{eq_sup:cor_non_cvx} and for this choice we have
	\begin{equation}\label{eq_sup:cor_non_cvx_optimal_gamma}
		\EE\left[\|\nabla f(\overline{z}^K)\|_\cD\right] \le \frac{\sqrt{2\left(f(x^0)-f(x^*)\right)L\gamma_\cD}}{\mu_\cD\sqrt{K}}.
	\end{equation}
\end{theorem}

\begin{proof}
	Taking full expectation from both sides of inequality \eqref{eq_sup:key_lemma} we get
	\begin{equation*}
		\EE\left[\|\nabla f(z^k)\|_\cD\right] \le \frac{(1-\beta)\EE\left[f(z^k) - f(z^{k+1})\right]}{\gamma\mu_\cD} + \frac{L\gamma\gamma_\cD}{2\mu_\cD(1-\beta)}.
	\end{equation*}
	 Further, summing up the results for $k=0,1,\ldots,K-1$, dividing both sides of the obtained inequality by $K$ and using tower property of the mathematical expectation we get
	\begin{equation*}
		\EE\left[\|\nabla f(\overline{z}^K)\|_\cD\right] = \frac{1}{K}\sum\limits_{k=0}^{K-1}\EE\left[\|\nabla f(z^k)\|_\cD\right] \le \frac{(1-\beta)(f(z^0) - f(x^*))}{K\gamma\mu_\cD} + \frac{L\gamma\gamma_\cD}{2\mu_\cD(1-\beta)}.
	\end{equation*}
	The last part where $\gamma = \frac{\gamma_0}{\sqrt{K}}$ is straightforward.
\end{proof}




\section{Convex Case}\label{sec_sup:cvx}

\begin{assumption}\label{ass_sup:bounded_level_set} 
	We assume that $f$ is convex, has a minimizer $x^*$ and has bounded level set at $x^0$:
	\begin{equation}\label{eq_sup:bounded_level_set}
		R_0 \eqdef \max\left\{\|x-x^*\|_\cD^* \mid f(x) \le f(x^0)\right\} < +\infty,
	\end{equation}
	where $\|\xi\|_\cD^* \eqdef \max\left\{\langle\xi,x\rangle\mid \|x\|_\cD \le 1\right\}$ defines the dual norm to $\|\cdot\|_\cD$.
\end{assumption} 


\begin{theorem}[Constant stepsize]\label{thm_sup:cvx_constant_stepsize}
	 Let Assumptions \ref{ass_sup:l_smoothness}, \ref{ass_sup:stp_general_assumption} and \ref{ass_sup:bounded_level_set} be satisfied. If we set $\gamma^k\equiv \gamma  < \frac{(1-\beta)R_0}{\mu_\cD}$, then for the iterates of {\tt SMTP} method the following inequality holds:
	\begin{equation}\label{eq_sup:cvx_constant_stepsize}
	 \EE\left[f(z^k)-f(x^*)\right] \le \left(1 - \frac{\gamma\mu_\cD}{(1-\beta)R_0}\right)^k\left(f(x^0)-f(x^*)\right) + \frac{L\gamma\gamma_\cD R_0}{2(1-\beta)\mu_\cD}.
	\end{equation}
	If we choose $\gamma = \frac{\varepsilon(1-\beta)\mu_\cD}{L\gamma_\cD R_0}$ for some $0<\varepsilon\le \frac{L\gamma_\cD R_0^2}{\mu_\cD^2}$ and run {\tt SMTP} for $k = K$ iterations where
	\begin{equation}\label{eq_sup:cvx_constant_stepsize_number_of_iterations}
		K = \frac{1}{\varepsilon}\frac{L\gamma_\cD R_0^2}{\mu_\cD^2}\ln\left(\frac{2(f(x^0) - f(x^*))}{\varepsilon}\right),
	\end{equation}
	then we will get $\EE\left[f(z^K)\right] - f(x^*) \le \varepsilon$.
\end{theorem}

\begin{proof}
	From the \eqref{eq_sup:key_lemma} and monotonicity of $\{f(z^k)\}_{k\ge 0}$ we have
	\begin{eqnarray*}
		\EE_{s^\sim\cD}\left[f(z^{k+1})\right] &\le& f(z^k) - \frac{\gamma\mu_\cD}{1-\beta}\|\nabla f(z^k)\|_\cD + \frac{L\gamma^2\gamma_\cD}{2(1-\beta)^2}\\
		&\overset{\eqref{eq:gradient_lower_bound}}{\le}& f(z^k) - \frac{\gamma\mu_\cD}{(1-\beta)R_0}(f(z^k)-f(x^*)) + \frac{L\gamma^2\gamma_\cD}{2(1-\beta)^2}.
	\end{eqnarray*}
	Taking full expectation, subtracting $f(x^*)$ from the both sides of the previous inequality and using the tower property of mathematical expectation we get
	\begin{equation}\label{eq_sup:cvx_recurrence}
		\EE\left[f(z^{k+1})-f(x^*)\right] \le \left(1 - \frac{\gamma\mu_\cD}{(1-\beta)R_0}\right)\EE\left[f(z^k)-f(x^*)\right] + \frac{L\gamma^2\gamma_\cD}{2(1-\beta)^2}.
	\end{equation}
	Since $\gamma < \frac{(1-\beta)R_0}{\mu_\cD}$ the term $1 - \frac{\gamma\mu_\cD}{(1-\beta)R_0}$ is positive and we can unroll the recurrence \eqref{eq_sup:cvx_recurrence}:
	\begin{eqnarray*}
		\EE\left[f(z^{k})-f(x^*)\right] &\le& \left(1 - \frac{\gamma\mu_\cD}{(1-\beta)R_0}\right)^k\left(f(z^0)-f(x^*)\right) + \frac{L\gamma^2\gamma_\cD}{2(1-\beta)^2}\sum\limits_{l=0}^{k-1}\left(1 - \frac{\gamma\mu_\cD}{(1-\beta)R_0}\right)^l\\
		&\le& \left(1 - \frac{\gamma\mu_\cD}{(1-\beta)R_0}\right)^k\left(f(x^0)-f(x^*)\right) + \frac{L\gamma^2\gamma_\cD}{2(1-\beta)^2}\sum\limits_{l=0}^{\infty}\left(1 - \frac{\gamma\mu_\cD}{(1-\beta)R_0}\right)^l\\
		&\le& \left(1 - \frac{\gamma\mu_\cD}{(1-\beta)R_0}\right)^k\left(f(x^0)-f(x^*)\right) + \frac{L\gamma^2\gamma_\cD}{2(1-\beta)^2}\cdot \frac{(1-\beta)R_0}{\gamma\mu_\cD}\\
		&=& \left(1 - \frac{\gamma\mu_\cD}{(1-\beta)R_0}\right)^k\left(f(x^0)-f(x^*)\right) + \frac{L\gamma\gamma_\cD R_0}{2(1-\beta)\mu_\cD}. 	
	\end{eqnarray*}
	Lastly, putting $\gamma = \frac{\varepsilon(1-\beta)\mu_\cD}{L\gamma_\cD R_0}$ and $k=K$ from \eqref{eq_sup:cvx_constant_stepsize_number_of_iterations} in \eqref{eq_sup:cvx_constant_stepsize} we have
	\begin{eqnarray*}
		\EE[f(z^K)] - f(x^*) &=& \left(1 - \frac{\varepsilon\mu_\cD^2}{L\gamma_\cD R_0^2}\right)^K\left(f(x^0) - f(x^*)\right) + \frac{\varepsilon}{2}\\
		&\le& \exp\left\{-K\cdot\frac{\varepsilon\mu_\cD^2}{L\gamma_\cD R_0^2}\right\}\left(f(x^0) - f(x^*)\right) + \frac{\varepsilon}{2}\\
		&\overset{\eqref{eq_sup:cvx_constant_stepsize_number_of_iterations}}{=}& \frac{\varepsilon}{2} + \frac{\varepsilon}{2} = \varepsilon. 
	\end{eqnarray*}	 
\end{proof}


Next we use technical lemma from \cite{mishchenko2019distributed}. We provide the original proof for completeness.

\begin{lem}[Lemma~6 from \cite{mishchenko2019distributed}]\label{lem:decreasing_stepsizes}
	Let a sequence $\{a^k\}_{k\ge 0}$ satisfy inequality $a^{k+1}\le (1 - \gamma^k\alpha) a^k + (\gamma^k)^2 N$ for any positive $\gamma^k \le \gamma_0$ with some constants $\alpha > 0, N>0, \gamma_0 > 0$. Further, let $\theta \ge \frac{2}{\gamma_0} $ and take $C$ such that $N\le \frac{\alpha\theta}{4}C$ and $a_0\le C$. Then, it holds
	\begin{align*}
		a^k \le \frac{ C}{\frac{\alpha}{\theta} k + 1}
	\end{align*}
	if we set $\gamma^k=\frac{2}{\alpha k + \theta}$.
\end{lem}
\begin{proof}
	We will show the inequality for $a^k$ by induction. Since inequality $a_0\le C$ is one of our assumptions, we have the initial step of the induction. To prove the inductive step, consider
	\begin{align*}
		a^{k+1} 
		\le (1 - \gamma^k\alpha) a^k + (\gamma^k)^2 N 
		\le \left(1 - \frac{2\alpha}{\alpha k + \theta} \right) \frac{\theta C}{\alpha k + \theta} + \theta\alpha\frac{C}{(\alpha k + \theta)^2}.
	\end{align*}
	To show that the right-hand side is upper bounded by $\frac{\theta C}{\alpha(k + 1) + \theta}$, one needs to have, after multiplying both sides by $(\alpha k + \theta)(\alpha k + \alpha + \theta)(\theta C)^{-1}$,
	\begin{align*}
		\left(1 - \frac{2\alpha}{\alpha k + \theta} \right) (\alpha k + \alpha + \theta) + \alpha\frac{\alpha k + \alpha + \theta}{\alpha k + \theta} \le \alpha k + \theta,
	\end{align*}
	which is equivalent to
	\begin{align*}
		\alpha - \alpha\frac{\alpha k + \alpha + \theta}{\alpha k + \theta} \le 0.
	\end{align*}
	The last inequality is trivially satisfied for all $k \ge 0$.
\end{proof}


\begin{theorem}[Decreasing stepsizes]\label{thm_sup:cvx_decreasing_stepsizes}
    Let Assumptions \ref{ass_sup:l_smoothness}, \ref{ass_sup:stp_general_assumption} and \ref{ass_sup:bounded_level_set} be satisfied. If we set $\gamma^k = \frac{2}{\alpha k + \theta}$, where $\alpha = \frac{\mu_\cD}{(1-\beta)R_0}$ and $\theta \ge \frac{2}{\alpha}$, then for the iterates of {\tt SMTP} method the following inequality holds:
   	\begin{equation}\label{eq_sup:cvx_decreasing_stepsizes}
   		\EE\left[f(z^k)\right] - f(x^*) \le \frac{1}{\eta k+1}\max\left\{ f(x^0) - f(x^*), \frac{2L\gamma_\cD}{\alpha\theta(1-\beta)^2}\right\},	
   	\end{equation}
    where $\eta\eqdef \frac{\alpha}{\theta}$. Then, if we choose $\gamma^k = \frac{2\alpha}{\alpha^2k+2}$ where $\alpha = \frac{\mu_\cD}{(1-\beta)R_0}$ and run {\tt SMTP} for $k = K$ iterations where
	\begin{equation}\label{eq_sup:cvx_decreasing_stepsizes_number_of_iterations}
		K = \frac{1}{\varepsilon}\cdot\frac{2R_0^2}{\mu_\cD^2}\max\left\{(1-\beta)^2(f(x^0)-f(x^*)), L\gamma_\cD\right\} - \frac{2(1-\beta)^2R_0^2}{\mu_\cD^2},\qquad \varepsilon > 0,
	\end{equation}
	 we get $\EE\left[f(z^K)\right] - f(x^*) \le \varepsilon$.
\end{theorem}

\begin{proof}
	In \eqref{eq_sup:cvx_recurrence} we proved that
	\begin{equation*}
		\EE\left[f(z^{k+1})-f(x^*)\right] \le \left(1 - \frac{\gamma\mu_\cD}{(1-\beta)R_0}\right)\EE\left[f(z^k)-f(x^*)\right] + \frac{L\gamma^2\gamma_\cD}{2(1-\beta)^2}.
	\end{equation*}
    Having that, we can apply Lemma~\ref{lem:decreasing_stepsizes} to the sequence $\EE\left[f(z^k)-f(x^*)\right]$. The constants for the lemma are: $N = \frac{L\gamma_\cD}{2(1-\beta)^2}$, $\alpha = \frac{\mu_\cD}{(1-\beta)R_0}$ and $C=\max\left\{f(x^0)-f(x^*), \frac{2L\gamma_\cD}{\alpha\theta(1-\beta)^2}\right\}$. Lastly, choosing $\gamma^k = \frac{2\alpha}{\alpha^2k+2}$ is equivalent to the choice $\theta = \frac{2}{\alpha}$. In this case, we have $\alpha\theta = 2$, $C = \max\left\{f(x^0) - f(x^*), \frac{L\gamma_\cD}{(1-\beta)^2}\right\}$ and $\eta = \frac{\alpha}{\theta} = \frac{\alpha^2}{2} = \frac{\mu_\cD^2}{2(1-\beta)^2R_0^2}$. Putting these parameters and $K$ from \eqref{eq_sup:cvx_decreasing_stepsizes_number_of_iterations} in the \eqref{eq_sup:cvx_decreasing_stepsizes} we get the result.
\end{proof}


\section{Strongly Convex Case}\label{sec_sup:str_cvx}
\begin{assumption}\label{ass_sup:strong_convexity}
	We assume that $f$ is $\mu$-strongly convex with respect to the norm {\bluetext $\|\cdot\|_\cD^*$}:
	\begin{equation}\label{eq_sup:strong_convexity}
	   f(y) \ge f(x) + \<\nabla f(x), y-x > + \frac{\mu}{2}{\bluetext\left(\|y-x\|_{\cD}^*\right)}^2,\quad \forall x,y\in\R^d.
	\end{equation}
\end{assumption}
It is well known that strong convexity implies
\begin{equation}\label{eq_sup:consequence_of_strong_convexity}
	\|\nabla f(x)\|_\cD^2 \ge 2\mu\left(f(x) - f(x^*)\right).
\end{equation}



\begin{theorem}[Solution-dependent stepsizes]\label{thm_up:str_cvx_sol_dep_stepsizes}
Let Assumptions \ref{ass_sup:l_smoothness}, \ref{ass_sup:stp_general_assumption} and \ref{ass_sup:strong_convexity} be satisfied. If we set $\gamma^k = \frac{(1-\beta)\theta_k\mu_\cD}{L}\sqrt{2\mu(f(z^k)-f(x^*))}$ for some $\theta_k\in(0,2)$ such that $\theta = \inf\limits_{k\ge 0}\{2\theta_k - \gamma_\cD\theta_k^2\} \in \left(0,\frac{L}{\mu_\cD^2\mu}\right)$, then for the iterates of {\tt SMTP} the following inequality holds:
	\begin{equation}\label{eq_sup:str_cvx_sol_dep_stepsizes}
		\EE\left[f(z^k)\right] - f(x^*) \le \left(1 - \frac{\theta\mu_\cD^2\mu}{L}\right)^{k}\left(f(x^0) - f(x^*)\right).
	\end{equation}
	If we run {\tt SMTP} for $k = K$ iterations where
	\begin{equation}\label{eq_sup:str_cvx_sol_dep_stepsizes_number_of_iterations}
		K = \frac{\kappa}{\theta\mu_\cD^2}\ln\left(\frac{f(x^0)-f(x^*)}{\varepsilon}\right),\qquad \varepsilon > 0,
	\end{equation}
	 where $\kappa\eqdef \frac{L}{\mu}$ is the condition number of the objective, we will get $\EE\left[f(z^K)\right] - f(x^*) \le \varepsilon$.
\end{theorem}

\begin{proof}
	From \eqref{eq_sup:key_lemma} and $\gamma^k = \frac{\theta_k\mu_\cD}{L}\sqrt{2\mu(f(x^k)-f(x^*))}$ we have
	\begin{eqnarray*}
		\EE_{s^k\sim\cD}\left[f(z^{k+1})\right] - f(x^*) &\le& f(z^k) - f(x^*) - \frac{\gamma^k\mu_\cD}{1-\beta}\|\nabla f(z^k)\|_\cD + \frac{L(\gamma^k)^2\gamma_\cD}{2(1-\beta)^2}\\
		&\overset{\eqref{eq_sup:consequence_of_strong_convexity}}{\le}& f(z^k) - f(x^*) - \frac{\gamma^k\mu_\cD}{1-\beta}\sqrt{2\mu(f(z^k)-f(x^*))}\\
		&&\quad + \frac{\gamma_\cD\theta_k^2\mu_\cD^2\mu}{L}(f(z^k)-f(x^*))\\
		&\le& f(z^k) - f(x^*) - \frac{2\theta^k\mu_\cD^2\mu}{L}(f(z^k)-f(x^*))\\
		&&\quad + \frac{\gamma_\cD\theta_k^2\mu_\cD^2\mu}{L}(f(z^k)-f(x^*))\\
		&\le& \left(1 - (2\theta_k - \gamma_\cD\theta_k^2)\frac{\mu_\cD^2\mu}{L}\right)(f(z^k)-f(x^*)).
	\end{eqnarray*}
	Using $\theta = \inf\limits_{k\ge 0}\{2\theta_k - \gamma_\cD\theta_k^2\} \in \left(0,\frac{L}{\mu_\cD^2\mu}\right)$ and taking the full expectation from the previous inequality we get
	\begin{eqnarray*}
		\EE\left[f(z^{k+1}) - f(x^*)\right] &\le& \left(1 - \frac{\theta\mu_\cD^2\mu}{L}\right)\EE\left[f(z^k) - f(x^*)\right]\\
		&\le& \left(1 - \frac{\theta\mu_\cD^2\mu}{L}\right)^{k+1}\left(f(x^0) - f(x^*)\right).
	\end{eqnarray*}
	 Lastly, from \eqref{eq_sup:str_cvx_sol_dep_stepsizes} we have
	\begin{eqnarray*}
		\EE\left[f(z^K)\right] - f(x^*) &\le& \left(1 - \frac{\theta\mu_\cD^2\mu}{L}\right)^{K}\left(f(x^0) - f(x^*)\right)\\
		&\le& \exp\left\{-K\frac{\theta\mu_\cD^2\mu}{L}\right\}\left(f(x^0) - f(x^*)\right)\\
		&\overset{\eqref{eq_sup:str_cvx_sol_dep_stepsizes_number_of_iterations}}{\le}& \varepsilon.
	\end{eqnarray*}
\end{proof}

\begin{assumption}\label{ass_sup:distr_unit_vectors}
	We assume that for all $s\sim\cD$ we have $\|s\|_2 = 1$.
\end{assumption}

\begin{theorem}[Solution-free stepsizes]\label{thm_sup:str_cvx_sol_free_stepsizes}
Let Assumptions \ref{ass_sup:l_smoothness}, \ref{ass_sup:stp_general_assumption}, \ref{ass_sup:strong_convexity} and \ref{ass_sup:distr_unit_vectors} be satisfied. If additionally we compute $f(z^k + ts^k)$, set $\gamma^k = \frac{(1-\beta)|f(z^k + ts^k) - f(z^k)|}{L t}$ for $t>0$ and assume that $\cD$ is such that $\mu_\cD^2 \le \frac{L}{\mu}$, then for the iterates of {\tt SMTP} the following inequality holds:
	\begin{equation}\label{eq_sup:str_cvx_sol_free_stepsizes}
		\EE\left[f(z^k)\right] - f(x^*) \le \left(1 - \frac{\mu_\cD^2\mu}{L}\right)^k\left(f(x^0) - f(x^*)\right) + \frac{L^2t^2}{8\mu_\cD^2\mu}.
	\end{equation}
	Moreover, for any $\varepsilon > 0$ if we set $t$ such that
	\begin{equation}\label{eq_sup:str_cvx_sol_free_stepsizes_t_bound}
		0 < t \le \sqrt{\frac{4\varepsilon\mu_\cD^2\mu}{L^2}},
	\end{equation}		
	and run {\tt SMTP} for $k = K$ iterations where
	\begin{equation}\label{eq_sup:str_cvx_sol_free_stepsizes_number_of_iterations}
		K = \frac{\kappa}{\mu_\cD^2}\ln\left(\frac{2(f(x^0)-f(x^*))}{\varepsilon}\right),
	\end{equation}
	 where $\kappa\eqdef \frac{L}{\mu}$ is the condition number of $f$, we will have $\EE\left[f(z^K)\right] - f(x^*) \le \varepsilon$.
\end{theorem}

\begin{proof}
	Recall that from \eqref{eq_sup:key_lemma_without_expectation} we have
	\begin{equation*}
		f(z^{k+1}) \le f(z^k) - \frac{\gamma^k}{1-\beta}|\langle\nabla f(z^k),s^k\rangle| + \frac{L(\gamma^k)^2}{2(1-\beta)^2}.
	\end{equation*}
	If we minimize the right hand side of the previous inequality as a function of $\gamma^k$, we will get that the optimal choice in this sense is $\gamma^k_{\text{opt}} = \frac{(1-\beta)|\langle\nabla f(z^k),s^k\rangle|}{L}$. However, this stepsize is impractical for derivative-free optimization, since it requires to know $\nabla f(z^k)$. The natural way to handle this is to approximate directional derivative $\langle\nabla f(z^k),s^k\rangle$ by finite difference $\frac{f(z^k+ts^k) - f(z^k)}{t}$ and that is what we do. We choose $\gamma^k = \frac{(1-\beta)|f(z^k + ts^k) - f(z^k)|}{L t} = \frac{(1-\beta)|\langle\nabla f(z^k),s^k\rangle|}{L} + \frac{(1-\beta)|f(z^k + ts^k) - f(z^k)|}{L t} - \frac{(1-\beta)|\langle\nabla f(z^k),s^k\rangle|}{L} \eqdef \gamma^k_{\text{opt}} + \delta^k$. From this we get
	\begin{eqnarray*}
		f(z^{k+1}) &\le& f(z^k) - \frac{|\langle\nabla f(z^k),s^k\rangle|^2}{2L} + \frac{L}{2(1-\beta)^2}(\delta^k)^2.
	\end{eqnarray*}
	Next we estimate $|\delta^k|$:
	\begin{eqnarray*}
		|\delta^k| &=& \frac{(1-\beta)}{L t}\left||f(z^k + ts^k) - f(z^k)| - |\langle\nabla f(z^k),ts^k\rangle|\right|\\
		&\le& \frac{(1-\beta)}{L t}\left|f(z^k + ts^k) - f(z^k) - \langle\nabla f(z^k),ts^k\rangle\right|\\
		&\overset{\eqref{eq:quadratic_upper_bound}}{\le}& \frac{(1-\beta)}{L t} \cdot\frac{L}{2}\|ts^k\|_2^2 = \frac{(1-\beta)t}{2}.
	\end{eqnarray*}
	It implies that
	\begin{eqnarray*}
		f(z^{k+1}) &\le& f(z^k) - \frac{|\langle\nabla f(z^k),s^k\rangle|^2}{2L} + \frac{L}{2(1-\beta)^2}\cdot\frac{(1-\beta)^2t^2}{4}\\
		&=& f(z^k) - \frac{|\langle\nabla f(z^k),s^k\rangle|^2}{2L}  + \frac{L t^2}{8}
	\end{eqnarray*}
	and after taking full expectation from the both sides of the obtained inequality we get
	\begin{equation*}
		\EE\left[f(z^{k+1}) - f(x^*)\right] \le \EE\left[f(z^k)-f(x^*)\right] - \frac{1}{2L}\EE\left[|\langle\nabla f(z^k), s^k\rangle|^2\right] + \frac{Lt^2}{8}.
	\end{equation*}
	Note that from the tower property of mathematical expectation and Jensen's inequality we have
	\begin{eqnarray*}
		\EE\left[|\langle\nabla f(z^k), s^k\rangle|^2\right] &=& \EE\left[\EE_{s^k\sim\cD}\left[|\langle\nabla f(z^k), s^k\rangle|^2\mid z^k\right]\right]\\
		&\ge& \EE\left[\left(\EE_{s^k\sim\cD}\left[|\langle\nabla f(z^k), s^k\rangle|\mid z^k\right]\right)^2\right]\\
		&\overset{\eqref{eq_sup:inner_product_lower_bound}}{\ge}& \EE\left[\mu_\cD^2\|\nabla f(z^k)\|_\cD^2\right] \overset{\eqref{eq_sup:consequence_of_strong_convexity}}{\ge} 2\mu_\cD^2\mu\EE\left[ f(z^k)-f(x^*)\right].
	\end{eqnarray*}
	Putting all together we get
	\begin{equation*}
		\EE\left[f(z^{k+1}) - f(x^*)\right] \le \left(1 - \frac{\mu_\cD^2\mu}{L}\right)\EE\left[f(z^k)-f(x^*)\right] + \frac{Lt^2}{8}.
	\end{equation*}
	Due to $\mu_\cD^2 \le \frac{L}{\mu}$ we have
	\begin{eqnarray*}
		\EE\left[f(z^k)  - f(x^*)\right] &\le& \left(1 - \frac{\mu_\cD^2\mu}{L}\right)^k\left(f(x^0) - f(x^*)\right) + \frac{Lt^2}{8}\sum\limits_{l=0}^{k-1} \left(1 - \frac{\mu_\cD^2\mu}{L}\right)^l\\
		&\le& \left(1 - \frac{\mu_\cD^2\mu}{L}\right)^k\left(f(x^0) - f(x^*)\right) + \frac{Lt^2}{8}\sum\limits_{l=0}^{\infty} \left(1 - \frac{\mu_\cD^2\mu}{L}\right)^l\\
		&=& \left(1 - \frac{\mu_\cD^2\mu}{L}\right)^k\left(f(x^0) - f(x^*)\right) + \frac{L^2t^2}{8\mu_\cD^2\mu}.
	\end{eqnarray*}
	Lastly, from \eqref{eq_sup:str_cvx_sol_free_stepsizes} we have
	\begin{eqnarray*}
		\EE\left[f(z^K)\right] - f(x^*) &\le& \left(1 - \frac{\mu_\cD^2\mu}{L}\right)^{K}\left(f(x^0) - f(x^*)\right) + \frac{L^2t^2}{8\mu_\cD^2\mu}\\
		&\overset{\eqref{eq_sup:str_cvx_sol_free_stepsizes_t_bound}}{\le}& \exp\left\{-K\frac{\mu_\cD^2\mu}{L}\right\}\left(f(x^0) - f(x^*)\right) + \frac{\varepsilon}{2}\\
		&\overset{\eqref{eq_sup:str_cvx_sol_free_stepsizes_number_of_iterations}}{\le}& \frac{\varepsilon}{2} + \frac{\varepsilon}{2} = \varepsilon.
	\end{eqnarray*}
\end{proof}

\newpage
\section{{\tt SMTP{\_}IS}: Stochastic Momentum Three Points with Importance Sampling}\label{sec:smtp_is}

Again by definition of $z^{k+1}$ we get that the sequence $\{f(z^k)\}_{k\ge 0}$ is monotone:
\begin{equation}\label{eq_sup:monotonicity_is}
	f(z^{k+1}) \le f(z^k) \qquad \forall k\ge 0.
\end{equation}

\begin{lem}\label{lem_sup:key_lemma_is}
	Assume that $f$ satisfies Assumption~\ref{as_sup:coord_wise_L_smoothness}. Then for the iterates of {\tt SMTP{\_}IS} the following inequalities hold:
	\begin{equation}\label{eq_sup:key_lemma_without_expectation_is}
		f(z^{k+1}) \le f(z^k) - \frac{\gamma_i^k}{1-\beta}|\nabla_{i_k} f(z^k)| + \frac{L_{i_k}(\gamma_i^k)^2}{2(1-\beta)^2}
	\end{equation}
	and
	\begin{equation}\label{eq_sup:key_lemma_is}
		\EE_{s^k\sim\cD}\left[f(z^{k+1})\right] \le f(z^k) - \frac{1}{1-\beta}\EE\left[\gamma_i^k|\nabla_{i_k} f(z^k)|\mid z^k\right] + \frac{1}{2(1-\beta)^2}\EE\left[L_{i_k}(\gamma_i^k)^2\mid z^k\right].
	\end{equation}	 
\end{lem}
\begin{proof}
	In the similar way as in Lemma~\ref{lem_sup:key_lemma} one can show that
	\begin{equation}\label{eq_sup:z_recursion_is}
		z^{k} = x^k - \frac{\gamma_i^k\beta}{1-\beta}v^{k-1}
	\end{equation}
	and
	\begin{eqnarray*}
		z_+^{k+1} = z^k - \frac{\gamma_i^k}{1-\beta}e_{i_k},
	\end{eqnarray*}
	\begin{eqnarray*}
		z_-^{k+1} = z^k + \frac{\gamma_i^k}{1-\beta}e_{i_k}.
	\end{eqnarray*}
	It implies that
	\begin{eqnarray*}
		f(z_+^{k+1}) &\overset{\eqref{eq_sup:coord_wise_L_smoothness}}{\le}& f(z^k) - \frac{\gamma_i^k}{1-\beta}\nabla_i f(z^k) + \frac{L_{i_k}(\gamma_i^k)^2}{2(1-\beta)^2}	
	\end{eqnarray*}
	and
	\begin{eqnarray*}
		f(z_-^{k+1}) &\le& f(z^k) + \frac{\gamma_i^k}{1-\beta}\nabla_i f(z^k) + \frac{L_{i_k}(\gamma_i^k)^2}{2(1-\beta)^2}.	
	\end{eqnarray*}
	Unifying these two inequalities we get
	\begin{eqnarray*}
		f(z^{k+1}) &\le& \min\{f(z_+^{k+1}),f(z_-^{k+1})\} = f(z^k) - \frac{\gamma_i^k}{1-\beta}|\nabla_i f(z^k)| + \frac{L_{i_k}(\gamma_i^k)^2}{2(1-\beta)^2},
	\end{eqnarray*}
	which proves \eqref{eq_sup:key_lemma_without_expectation_is}.
	Finally, taking the expectation $\EE[\cdot\mid z^k]$ conditioned on $z^k$ from the both sides of the previous inequality we obtain
	\begin{equation*}
		\EE\left[f(z^{k+1})\mid z^k\right] \le  f(z^k) - \frac{1}{1-\beta}\EE\left[\gamma_i^k|\nabla_{i_k} f(z^k)|\mid z^k\right] + \frac{1}{2(1-\beta)^2}\EE\left[L_{i_k}(\gamma_i^k)^2\mid z^k\right].
	\end{equation*}
\end{proof}

\subsection{Non-convex Case}\label{sec:non_cvx_is}
\begin{theorem}\label{thm:non_cvx_is}
	Assume that $f$ satisfies Assumption~\ref{as_sup:coord_wise_L_smoothness}. Let {\tt SMTP{\_}IS} with $\gamma_i^k = \frac{\gamma}{w_{i_k}}$ for some $\gamma > 0$ produce points $\{z^0,z^1,\ldots, z^{K-1}\}$ and $\overline{z}^K$ is chosen uniformly at random among them. Then
	\begin{equation}\label{eq_sup:non_cvx_is}
	\EE\left[\|\nabla f(\overline{z}^K)\|_1\right] \le \frac{(1-\beta)(f(x^0) - f(x^*))}{K\gamma\min\limits_{i=1,\ldots,d}\frac{p_i}{w_i}} + \frac{\gamma}{2(1-\beta)\min\limits_{i=1,\ldots,d}\frac{p_i}{w_i}}\sum\limits_{i=1}^{d}\frac{L_ip_i}{w_i^2}.
	\end{equation}
	Moreover, if we choose $\gamma = \frac{\gamma_0}{\sqrt{K}}$, then
	\begin{equation}\label{eq_sup:cor_non_cvx_is}
		\EE\left[\|\nabla f(\overline{z}^K)\|_1\right] \le \frac{1}{\sqrt{K}\min\limits_{i=1,\ldots,d}\frac{p_i}{w_i}}\left(\frac{(1-\beta)(f(x^0) - f(x^*))}{\gamma_0} + \frac{\gamma_0}{2(1-\beta)}\sum\limits_{i=1}^{d}\frac{L_ip_i}{w_i^2}\right).
	\end{equation}
	Note that if we choose $\gamma_0 = \sqrt{\frac{2(1-\beta)^2(f(x^0) - f(x^*))}{\sum\limits_{i=1}^d\frac{L_i p_i}{w_i^2}}}$ in order to minimize right-hand side of \eqref{eq_sup:cor_non_cvx_is}, we will get
\begin{equation}\label{eq_sup:cor_non_cvx_optimal_gamma_is}
		\EE\left[\|\nabla f(\overline{z}^K)\|_1\right] \le \frac{\sqrt{2\left(f(x^0)-f(x^*)\right)\sum\limits_{i=1}^d\frac{L_ip_i}{w_i^2}}}{\sqrt{K}\min\limits_{i=1,\ldots,d}\frac{p_i}{w_i}}.
	\end{equation}
	Note that for $p_i=\nicefrac{L_i}{\sum_i^d L_i}$ with $w_i = L_i$ we have that the rates improves to
	\begin{equation}
		\EE\left[\|\nabla f(\overline{z}^K)\|_1\right] \le \frac{\sqrt{2(f(x^0) - f(x^*)) d\sum_{i=1}^dL_i}}{\sqrt{K}}.
	\end{equation}
\end{theorem}
\begin{proof}
	Recall that from \eqref{eq_sup:key_lemma_is} we have
	\begin{equation}\label{eq_sup:non_cvx_technical_is}
		\EE\left[f(z^{k+1})\mid z^k\right] \le  f(z^k) - \frac{1}{1-\beta}\EE\left[\gamma_i^k|\nabla_{i_k} f(z^k)|\mid z^k\right] + \frac{1}{2(1-\beta)^2}\EE\left[L_{i_k}(\gamma_i^k)^2\mid z^k\right].
	\end{equation}
	Using our choice $\gamma_i^k = \frac{\gamma}{w_{i_k}}$ we derive
	\begin{eqnarray*}
		\EE\left[\gamma_i^k|\nabla_{i_k} f(z^k)|\mid z^k\right] = \gamma\sum\limits_{i=1}^d \frac{p_i}{w_i}|\nabla_i f(z^k)|\ge \gamma\|\nabla f(z^k)\|_1 \min\limits_{i=1,\ldots,d}\frac{p_i}{w_i}
	\end{eqnarray*}
	and
	\begin{eqnarray*}
		\EE\left[L_{i_k}(\gamma_i^k)^2\mid z^k\right] = \gamma^2\sum\limits_{i=1}^d\frac{L_i p_i}{w_i^2}.
	\end{eqnarray*}
	Putting it in \eqref{eq_sup:non_cvx_technical_is} and taking full expectation from the both sides of obtained inequality we get
	\begin{eqnarray*}
		\EE\left[f(z^{k+1})\right] \le \EE\left[f(z^k)\right] - \frac{\gamma\min\limits_{i=1,\ldots,d}\frac{p_i}{w_i}}{1-\beta}\EE\|\nabla f(z^k)\|_1 + \frac{\gamma^2}{2(1-\beta)^2}\sum\limits_{i=1}^{d}\frac{L_ip_i}{w_i^2},
	\end{eqnarray*}
	whence
	\begin{eqnarray*}
		\|\nabla f(z^k)\|_1 \le \frac{(1-\beta)\left(\EE\left[f(z^k)\right] - \EE\left[f(z^{k+1})\right]\right)}{\gamma\min\limits_{i=1,\ldots,d}\frac{p_i}{w_i}} + \frac{\gamma}{2(1-\beta)\min\limits_{i=1,\ldots,d}\frac{p_i}{w_i}}\sum\limits_{i=1}^{d}\frac{L_ip_i}{w_i^2}.
	\end{eqnarray*}
	Summing up previous inequality for $k=0,1,\ldots,K-1$ and dividing both sides of the result by $K$, we get
	\begin{equation*}
		\frac{1}{K}\sum\limits_{k=0}^{K-1}\EE\left[\|\nabla f(z^k)\|_1\right] \le \frac{(1-\beta)(f(z^0) - f(x^*))}{K\gamma\min\limits_{i=1,\ldots,d}\frac{p_i}{w_i}} + \frac{\gamma}{2(1-\beta)\min\limits_{i=1,\ldots,d}\frac{p_i}{w_i}}\sum\limits_{i=1}^{d}\frac{L_ip_i}{w_i^2}.
	\end{equation*}
	It remains to notice that $\frac{1}{K}\sum\limits_{k=0}^{K-1}\EE\left[\|\nabla f(z^k)\|_1\right] = \EE\left[\|\nabla f(\overline{z}^K)\|_1\right]$.
	The last part where $\gamma = \frac{\gamma_0}{\sqrt{K}}$ is straightforward.
\end{proof}

	
\subsection{Convex Case}\label{sec:cvx_is} 
As for {\tt SMTP} to tackle convex problems by {\tt SMTP{\_}IS} we use Assumption~\ref{as:bounded_level_set} with $\|\cdot\|_\cD = \|\cdot\|_1$. Note that in this case $R_0 = \max\left\{\|x - x^*\|_\infty\mid f(x) \le f(x^0)\right\}$.

\begin{theorem}[Constant stepsize]\label{thm:cvx_constant_stepsize_is}
	Let Assumptions~\ref{as:bounded_level_set}~and~\ref{as_sup:coord_wise_L_smoothness} be satisfied. If we set $\gamma_i^k = \frac{\gamma}{w_{i_k}}$ such that $0<\gamma\le \frac{(1-\beta)R_0}{\min\limits_{i=1,\ldots,d}\frac{p_i}{w_i}}$, then for the iterates of {\tt SMTP{\_}IS} method the following inequality holds:
	\begin{equation}\label{eq_sup:cvx_constant_stepsize_is}
	 \EE\left[f(z^k)-f(x^*)\right] \le \left(1 - \frac{\gamma\min\limits_{i=1,\ldots,d}\frac{p_i}{w_i}}{(1-\beta)R_0}\right)^k\left(f(z^0)-f(x^*)\right) + \frac{\gamma R_0}{2(1-\beta)\min\limits_{i=1,\ldots,d}\frac{p_i}{w_i}}\sum\limits_{i=1}^d\frac{L_i p_i}{w_i^2}.
	\end{equation}
	Moreover, if we choose $\gamma = \frac{\varepsilon(1-\beta)\min\limits_{i=1,\ldots,d}\frac{p_i}{w_i}}{R_0\sum\limits_{i=1}^d \frac{L_ip_i}{w_i^2}}$ for some $0<\varepsilon\le \frac{R_0^2\sum\limits_{i=1}^d \frac{L_ip_i}{w_i^2}}{\min\limits_{i=1,\ldots,d}\frac{p_i^2}{w_i^2}}$ and run {\tt SMTP{\_}IS} for $k = K$ iterations where
	\begin{equation}\label{eq_sup:cvx_constant_stepsize_number_of_iterations_is}
		K = \frac{1}{\varepsilon}\frac{R_0^2\sum\limits_{i=1}^d\frac{L_ip_i}{w_i^2}}{\min\limits_{i=1,\ldots,d}\frac{p_i^2}{w_i^2}}\ln\left(\frac{2(f(x^0) - f(x^*))}{\varepsilon}\right),
	\end{equation}
	we will get $\EE\left[f(z^K)\right] - f(x^*) \le \varepsilon$. Moreover, for $p_i=\nicefrac{L_i}{\sum_i^d L_i}$ with $w_i = L_i$, the rate improves to
	\begin{equation}
		K = \frac{1}{\epsilon} R_0^2 d\sum_{i=1}^d L_i \ln\left(\frac{2(f(x^0) - f(x^*))}{\varepsilon}\right).
	\end{equation}
\end{theorem}
\begin{proof}
	Recall that from \eqref{eq_sup:key_lemma_is} we have
	\begin{equation}\label{eq_sup:cvx_technical_is}
		\EE\left[f(z^{k+1})\mid z^k\right] \le  f(z^k) - \frac{1}{1-\beta}\EE\left[\gamma_i^k|\nabla_{i_k} f(z^k)|\mid z^k\right] + \frac{1}{2(1-\beta)^2}\EE\left[L_{i_k}(\gamma_i^k)^2\mid z^k\right].
	\end{equation}
	Using our choice $\gamma_i^k = \frac{\gamma}{w_{i_k}}$ we derive
	\begin{eqnarray*}
		\EE\left[\gamma_i^k\nabla_{i_k} f(z^k)\mid z^k\right] &=& \gamma\sum\limits_{i=1}^d \frac{p_i}{w_i}|\nabla_i f(z^k)|\ge \gamma\|\nabla f(z^k)\|_1 \min\limits_{i=1,\ldots,d}\frac{p_i}{w_i}\\
		&\overset{\eqref{eq:gradient_lower_bound}}{\ge}& \frac{\gamma}{R_0}\min\limits_{i=1,\ldots,d}\frac{p_i}{w_i}\left(f(z^k)-f(x^*)\right)
	\end{eqnarray*}
	and
	\begin{eqnarray*}
		\EE\left[L_{i_k}(\gamma_i^k)^2\mid z^k\right] = \gamma^2\sum\limits_{i=1}^d\frac{L_i p_i}{w_i^2}.
	\end{eqnarray*}
	Putting it in \eqref{eq_sup:cvx_technical_is} and taking full expectation from the both sides of obtained inequality we get
	\begin{equation}\label{eq_sup:cvx_recurrence_is}
		\EE\left[f(z^{k+1})-f(x^*)\right] \le \left(1 - \frac{\gamma\min\limits_{i=1,\ldots,d}\frac{p_i}{w_i}}{(1-\beta)R_0}\right)\EE\left[f(z^k)-f(x^*)\right] + \frac{\gamma^2}{2(1-\beta)^2}\sum\limits_{i=1}^d\frac{L_i p_i}{w_i^2}.
	\end{equation}
	Due to our choice of $\gamma \le \frac{(1-\beta)R_0}{\min\limits_{i=1,\ldots,d}\frac{p_i}{w_i}}$ we have that the factor $\left(1 - \frac{\gamma}{(1-\beta)R_0}\min\limits_{i=1,\ldots,d}\frac{p_i}{w_i}\right)$ is non-negative and, therefore,
	\begin{eqnarray*}
		\EE\left[f(z^{k})-f(x^*)\right] &\le& \left(1 - \frac{\gamma}{(1-\beta)R_0}\min\limits_{i=1,\ldots,d}\frac{p_i}{w_i}\right)^k\left(f(z^0)-f(x^*)\right)\\
		&&\quad + \left(\frac{\gamma^2}{2(1-\beta)^2}\sum\limits_{i=1}^d\frac{L_i p_i}{w_i^2}\right)\sum\limits_{l=0}^{k-1}\left(1 - \frac{\gamma}{(1-\beta)R_0}\min\limits_{i=1,\ldots,d}\frac{p_i}{w_i}\right)^l\\
		&\le& \left(1 - \frac{\gamma}{(1-\beta)R_0}\min\limits_{i=1,\ldots,d}\frac{p_i}{w_i}\right)^k\left(f(z^0)-f(x^*)\right)\\
		&&\quad + \left(\frac{\gamma^2}{2(1-\beta)^2}\sum\limits_{i=1}^d\frac{L_i p_i}{w_i^2}\right)\sum\limits_{l=0}^{\infty}\left(1 - \frac{\gamma}{(1-\beta)R_0}\min\limits_{i=1,\ldots,d}\frac{p_i}{w_i}\right)^l\\
		&\le& \left(1 - \frac{\gamma\min\limits_{i=1,\ldots,d}\frac{p_i}{w_i}}{(1-\beta)R_0}\right)^k\left(f(z^0)-f(x^*)\right) + \frac{\gamma R_0}{2(1-\beta)\min\limits_{i=1,\ldots,d}\frac{p_i}{w_i}}\sum\limits_{i=1}^d\frac{L_i p_i}{w_i^2}.
	\end{eqnarray*}	 
	Then, putting $\gamma = \frac{\varepsilon(1-\beta)\min\limits_{i=1,\ldots,d}\frac{p_i}{w_i}}{R_0\sum\limits_{i=1}^d \frac{L_ip_i}{w_i^2}}$ and $k=K$ from \eqref{eq_sup:cvx_constant_stepsize_number_of_iterations_is} in \eqref{eq_sup:cvx_constant_stepsize_is} we have
	\begin{eqnarray*}
		\EE[f(z^K)] - f(x^*) &=& \left(1 - \frac{\varepsilon\min\limits_{i=1,\ldots,d}\frac{p_i^2}{w_i^2}}{R_0^2\sum\limits_{i=1}^d\frac{L_ip_i}{w_i^2}}\right)^K\left(f(z^0)-f(x^*)\right) + \frac{\varepsilon}{2}\\
		&\le& \exp\left\{-K\cdot\frac{\varepsilon\min\limits_{i=1,\ldots,d}\frac{p_i^2}{w_i^2}}{R_0^2\sum\limits_{i=1}^d\frac{L_ip_i}{w_i^2}}\right\}\left(f(z^0)-f(x^*)\right) + \frac{\varepsilon}{2}\\
		&\overset{\eqref{eq_sup:cvx_constant_stepsize_number_of_iterations_is}}{=}& \frac{\varepsilon}{2} + \frac{\varepsilon}{2} = \varepsilon. 
	\end{eqnarray*}	 
\end{proof}

\begin{theorem}[Decreasing stepsizes]\label{thm:cvx_decreasing_stepsizes_is}
    Let Assumptions~\ref{as:bounded_level_set}~and~\ref{as_sup:coord_wise_L_smoothness} be satisfied. If we set $\gamma_i^k = \frac{\gamma^k}{w_{i_k}}$ and $\gamma^k = \frac{2}{\alpha k + \theta}$, where $\alpha = \frac{\min\limits_{i=1,\ldots,d}\frac{p_i}{w_i}}{(1-\beta)R_0}$ and $\theta \ge \frac{2}{\alpha}$, then for the iterates of {\tt SMTP{\_}IS} method the following inequality holds:
   	\begin{equation}\label{eq_sup:cvx_decreasing_stepsizes_is}
   		\EE\left[f(z^k)\right] - f(x^*) \le \frac{1}{\eta k+1}\max\left\{f(x^0)-f(x^*), \frac{2}{\alpha\theta(1-\beta)^2}\sum\limits_{i=1}^d\frac{L_ip_i}{w_i^2}\right\},	
   	\end{equation}
    where $\eta\eqdef \frac{\alpha}{\theta}$. Moreover, if we choose $\gamma^k = \frac{2\alpha}{\alpha^2k+2}$ where $\alpha = \frac{\min\limits_{i=1,\ldots,d}\frac{p_i}{w_i}}{(1-\beta)R_0}$ and run {\tt SMTP{\_}IS} for $k = K$ iterations where
	\begin{equation}\label{eq_sup:cvx_decreasing_stepsizes_number_of_iterations_is}
		K = \frac{1}{\varepsilon}\cdot\frac{2R_0^2}{\min\limits_{i=1,\ldots,d}\frac{p_i^2}{w_i^2}}\max\left\{(1-\beta)^2(f(x^0)-f(x^*)), \sum\limits_{i=1}^d\frac{L_ip_i}{w_i^2}\right\} - \frac{2(1-\beta)^2R_0^2}{\min\limits_{i=1,\ldots,d}\frac{p_i^2}{w_i^2}},\qquad \varepsilon > 0,
	\end{equation}
	 we will get $\EE\left[f(z^K)\right] - f(x^*) \le \varepsilon$.
\end{theorem}
\begin{proof}
	In \eqref{eq_sup:cvx_recurrence_is} we proved that
	\begin{equation*}
		\EE\left[f(z^{k+1})-f(x^*)\right] \le \left(1 - \frac{\gamma\min\limits_{i=1,\ldots,d}\frac{p_i}{w_i}}{(1-\beta)R_0}\right)\EE\left[f(z^k)-f(x^*)\right] + \frac{\gamma^2}{2(1-\beta)^2}\sum\limits_{l=1}^d\frac{L_ip_i}{w_i^2}.
	\end{equation*}
    Having that, we can apply Lemma~\ref{lem:decreasing_stepsizes} to the sequence $\EE\left[f(z^k)-f(x^*)\right]$. The constants for the lemma are: $N = \frac{1}{2(1-\beta)^2}\sum\limits_{l=1}^d\frac{L_ip_i}{w_i^2}$, $\alpha = \frac{\min\limits_{i=1,\ldots,d}\frac{p_i}{w_i}}{(1-\beta)R_0}$ and $C=\max\left\{f(x^0)-f(x^*), \frac{2}{\alpha\theta(1-\beta)^2}\sum\limits_{i=1}^d\frac{L_ip_i}{w_i^2}\right\}$.
    Lastly, note that choosing $\gamma^k = \frac{2\alpha}{\alpha^2k+2}$ is equivalent to choice $\theta = \frac{2}{\alpha}$. In this case we have $\alpha\theta = 2$ and $C = \max\left\{f(x^0) - f(x^*), \frac{1}{(1-\beta)^2}\sum\limits_{i=1}^d\frac{L_ip_i}{w_i^2}\right\}$ and $\eta = \frac{\alpha}{\theta} = \frac{\alpha^2}{2} = \frac{\min\limits_{i=1,\ldots,d}\frac{p_i^2}{w_i^2}}{2(1-\beta)^2R_0^2}$. Putting these parameters and $K$ from \eqref{eq_sup:cvx_decreasing_stepsizes_number_of_iterations_is} in the \eqref{eq_sup:cvx_decreasing_stepsizes_is} we get the result.
\end{proof}


\subsection{Strongly Convex Case}\label{sec_sup:str_cvx_is}

\begin{theorem}[Solution-dependent stepsizes]\label{thm_sup:str_cvx_sol_dep_stepsizes_is}
	Let Assumptions~\ref{as:strong_convexity} (with $\|\cdot\|_\cD = \|\cdot\|_1$) and~\ref{as_sup:coord_wise_L_smoothness} be satisfied. If we set $\gamma_i^k = \frac{(1-\beta)\theta_k\min\limits_{i=1,\ldots,d}\frac{p_i}{w_i}}{w_{i_k}\sum\limits_{i=1}^d\frac{L_ip_i}{w_i^2}}\sqrt{2\mu(f(z^k)-f(x^*))}$ for some $\theta_k\in(0,2)$ such that $\theta = \inf\limits_{k\ge 0}\{2\theta_k - \theta_k^2\} \in \left(0,\frac{\sum\limits_{i=1}^d\frac{L_ip_i}{w_i^2}}{\mu\min\limits_{i=1,\ldots,d}\frac{p_i^2}{w_i^2}}\right)$, then for the iterates of {\tt SMTP{\_}IS} method the following inequality holds:
	\begin{equation}\label{eq_sup:str_cvx_sol_dep_stepsizes_is}
		\EE\left[f(z^k)\right] - f(x^*) \le \left(1 - \frac{\theta\mu\min\limits_{i=1,\ldots,d}\frac{p_i^2}{w_i^2}}{\sum\limits_{i=1}^d\frac{L_ip_i}{w_i^2}}\right)^{k}\left(f(x^0) - f(x^*)\right).
	\end{equation}
	If we run {\tt SMTP{\_}IS} for $k = K$ iterations where
	\begin{equation}\label{eq_sup:str_cvx_sol_dep_stepsizes_number_of_iterations_is}
		K = \frac{\sum\limits_{i=1}^d\frac{L_ip_i}{w_i^2}}{\theta\mu\min\limits_{i=1,\ldots,d}\frac{p_i^2}{w_i^2}}\ln\left(\frac{f(x^0)-f(x^*)}{\varepsilon}\right),\qquad \varepsilon > 0,
	\end{equation}
	 we will get $\EE\left[f(z^K)\right] - f(x^*) \le \varepsilon$.
\end{theorem}
\begin{proof}
	Recall that from \eqref{eq_sup:key_lemma_is} we have
	\begin{equation}\label{eq_sup:str_cvx_technical_is}
		\EE\left[f(z^{k+1})\mid z^k\right] \le  f(z^k) - \frac{1}{1-\beta}\EE\left[\gamma_i^k|\nabla_{i_k} f(z^k)|\mid z^k\right] + \frac{1}{2(1-\beta)^2}\EE\left[L_{i_k}(\gamma_i^k)^2\mid z^k\right].
	\end{equation}
	Using our choice $\gamma_i^k = \frac{(1-\beta)\theta_k\min\limits_{i=1,\ldots,d}\frac{p_i}{w_i}}{w_{i_k}\sum\limits_{i=1}^d\frac{L_ip_i}{w_i^2}}\sqrt{2\mu(f(z^k)-f(x^*))}$ we derive
	\begin{eqnarray*}
		\EE\left[\gamma_i^k\nabla_{i_k} f(z^k)\mid z^k\right] &=& \frac{(1-\beta)\theta_k\min\limits_{i=1,\ldots,d}\frac{p_i}{w_i}}{\sum\limits_{i=1}^d\frac{L_ip_i}{w_i^2}}\sqrt{2\mu(f(z^k)-f(x^*))}\sum\limits_{i=1}^d \frac{p_i}{w_i}|\nabla_i f(z^k)|\\
		&\ge& \frac{(1-\beta)\theta_k\left(\min\limits_{i=1,\ldots,d}\frac{p_i}{w_i}\right)^2}{\sum\limits_{i=1}^d\frac{L_ip_i}{w_i^2}}\sqrt{2\mu(f(z^k)-f(x^*))}\|\nabla f(z^k)\|_1\\
		&\overset{\eqref{eq:consequence_of_strong_convexity}}{\ge}& \frac{2(1-\beta)\theta_k\min\limits_{i=1,\ldots,d}\frac{p_i^2}{w_i^2}}{\sum\limits_{i=1}^d\frac{L_ip_i}{w_i^2}}\mu(f(z^k)-f(x^*))
	\end{eqnarray*}
	and
	\begin{eqnarray*}
		\EE\left[L_{i_k}(\gamma_i^k)^2\mid z^k\right] &=& \frac{2(1-\beta)^2\theta_k^2\min\limits_{i=1,\ldots,d}\frac{p_i^2}{w_i^2}}{\left(\sum\limits_{i=1}^d\frac{L_ip_i}{w_i^2}\right)^2}\mu(f(z^k)-f(x^*))\sum\limits_{i=1}^d\frac{L_i p_i}{w_i^2}\\
		&=& \frac{2(1-\beta)^2\theta_k^2\min\limits_{i=1,\ldots,d}\frac{p_i^2}{w_i^2}}{\sum\limits_{i=1}^d\frac{L_ip_i}{w_i^2}}\mu(f(z^k)-f(x^*)).
	\end{eqnarray*}
	Putting it in \eqref{eq_sup:str_cvx_technical_is} and taking full expectation from the both sides of obtained inequality we get
	\begin{eqnarray*}
		\EE\left[f(z^{k+1})-f(x^*)\right] \le \left(1 - (2\theta - \theta^2)\frac{\mu\min\limits_{i=1,\ldots,d}\frac{p_i^2}{w_i^2}}{\sum\limits_{i=1}^d\frac{L_ip_i}{w_i^2}}\right)\EE\left[f(z^k)-f(x^*)\right].
	\end{eqnarray*}		
	Using $\theta = \inf\limits_{k\ge 0}\{2\theta_k - \theta_k^2\} \in \left(0,\frac{\sum\limits_{i=1}^d\frac{L_ip_i}{w_i^2}}{\mu\min\limits_{i=1,\ldots,d}\frac{p_i^2}{w_i^2}}\right)$ we obtain
	\begin{eqnarray*}
		\EE\left[f(z^{k+1}) - f(x^*)\right] &\le& \left(1 - \frac{\theta\mu\min\limits_{i=1,\ldots,d}\frac{p_i^2}{w_i^2}}{\sum\limits_{i=1}^d\frac{L_ip_i}{w_i^2}}\right)\EE\left[f(z^k) - f(x^*)\right]\\
		&\le& \left(1 - \frac{\theta\mu\min\limits_{i=1,\ldots,d}\frac{p_i^2}{w_i^2}}{\sum\limits_{i=1}^d\frac{L_ip_i}{w_i^2}}\right)^{k+1}\left(f(x^0) - f(x^*)\right).
	\end{eqnarray*}
	Lasrtly, from \eqref{eq_sup:str_cvx_sol_dep_stepsizes_is} we have
	\begin{eqnarray*}
		\EE\left[f(z^K)\right] - f(x^*) &\le& \left(1 - \frac{\theta\mu\min\limits_{i=1,\ldots,d}\frac{p_i^2}{w_i^2}}{\sum\limits_{i=1}^d\frac{L_ip_i}{w_i^2}}\right)^{K}\left(f(x^0) - f(x^*)\right)\\
		&\le& \exp\left\{-K\frac{\theta\mu\min\limits_{i=1,\ldots,d}\frac{p_i^2}{w_i^2}}{\sum\limits_{i=1}^d\frac{L_ip_i}{w_i^2}}\right\}\left(f(x^0) - f(x^*)\right)\\
		&\overset{\eqref{eq_sup:str_cvx_sol_dep_stepsizes_number_of_iterations_is}}{\le}& \varepsilon.
	\end{eqnarray*}
\end{proof}

The previous result based on the choice of $\gamma^k$ which depends on the $f(z^k) - f(x^*)$ which is often unknown in practice. The next theorem does not have this drawback and makes it possible to obtain the same rate of convergence as in the previous theorem using one extra function evaluation.

\begin{theorem}[Solution-free stepsizes]\label{thm:str_cvx_sol_free_stepsizes_is}
	Let Assumptions~\ref{as:strong_convexity} (with $\|\cdot\|_\cD = \|\cdot\|_2$) and~\ref{as_sup:coord_wise_L_smoothness} be satisfied. If additionally we compute $f(z^k + te_{i_k})$, set $\gamma_i^k = \frac{(1-\beta)|f(z^k + te_{i_k}) - f(z^k)|}{L_{i_k} t}$ for $t>0$, then for the iterates of {\tt SMTP{\_}IS} method the following inequality holds:
	\begin{equation}\label{eq_sup:str_cvx_sol_free_stepsizes_is}
		\EE\left[f(z^k)\right] - f(x^*) \le \left(1 - \mu\min\limits_{i=1,\ldots,d}\frac{p_i}{L_i}\right)^k\left(f(x^0) - f(x^*)\right) + \frac{t^2}{8\mu\min\limits_{i=1,\ldots,d}\frac{p_i}{L_i}}\sum\limits_{i=1}^dp_iL_i.
	\end{equation}
	Moreover, for any $\varepsilon > 0$ if we set $t$ such that
	\begin{equation}\label{eq_sup:str_cvx_sol_free_stepsizes_t_bound_is}
		0 < t \le \sqrt{\frac{4\varepsilon\mu\min\limits_{l=1,\ldots,d}\frac{p_i}{L_i}}{\sum\limits_{i=1}^dp_iL_i}},
	\end{equation}		
	 and run {\tt SMTP{\_}IS} for $k = K$ iterations where
	\begin{equation}\label{eq_sup:str_cvx_sol_free_stepsizes_number_of_iterations_is}
		K = \frac{1}{\mu\min\limits_{i=1,\ldots,d}\frac{p_i}{L_i}}\ln\left(\frac{2(f(x^0)-f(x^*))}{\varepsilon}\right),
	\end{equation}
	 we will get $\EE\left[f(z^K)\right] - f(x^*) \le \varepsilon$.
	 Moreover, note that for $p_i=\nicefrac{L_i}{\sum_i^d L_i}$ with $w_i = L_i$, the rate improves to
	\begin{equation}
		K = \frac{\sum_{i=1}^d L_i}{\mu}\ln\left(\frac{2(f(x^0) - f(x^*))}{\varepsilon}\right).
	\end{equation}
\end{theorem}
\begin{proof}
	Recall that from \eqref{eq_sup:key_lemma_without_expectation_is} we have
	\begin{equation*}
		f(z^{k+1}) \le f(z^k) - \frac{\gamma_i^k}{1-\beta}|\nabla_{i_k} f(z^k)| + \frac{L_{i_k}(\gamma_i^k)^2}{2(1-\beta)^2}.
	\end{equation*}
	If we minimize the right hand side of the previous inequality as a function of $\gamma_i^k$, we will get that the optimal choice in this sense is $\gamma^k_{\text{opt}} = \frac{(1-\beta)|\nabla_{i_k} f(z^k)|}{L_{i_k}}$. However, this stepsize is impractical for derivative-free optimization, since it requires to know $\nabla_{i_k} f(z^k)$. The natural way to handle this is to approximate directional derivative $\nabla_{i_k} f(z^k)$ by finite difference $\frac{f(z^k+te_{i_k}) - f(z^k)}{t}$ and that is what we do. We choose $\gamma_i^k = \frac{(1-\beta)|f(z^k + te_{i_k}) - f(z^k)|}{L_{i_k} t} = \frac{(1-\beta)|\nabla_{i_k} f(z^k)|}{L_{i_k}} + \frac{(1-\beta)|f(z^k + te_{i_k}) - f(z^k)|}{L_{i_k} t} - \frac{(1-\beta)|\nabla_{i_k} f(z^k)|}{L_{i_k}} \eqdef \gamma^k_{\text{opt}} + \delta_i^k$. From this we get
	\begin{eqnarray*}
		f(z^{k+1}) &\le& f(z^k) - \frac{|\nabla_{i_k} f(z^k)|^2}{2L_{i_k}} + \frac{L_{i_k}}{2(1-\beta)^2}(\delta_i^k)^2.
	\end{eqnarray*}
	Next we estimate $|\delta_i^k|$:
	\begin{eqnarray*}
		|\delta_i^k| &=& \frac{(1-\beta)}{L_{i_k} t}\left||f(z^k + te_{i_k}) - f(z^k)| - |\nabla_{i_k} f(z^k)|t\right|\\
		&\le& \frac{(1-\beta)}{L_{i_k} t}\left|f(z^k + te_{i_k}) - f(z^k) - \nabla_{i_k} f(z^k)t\right|\\
		&\overset{\eqref{eq_sup:coord_wise_L_smoothness}}{\le}& \frac{(1-\beta)}{L_{i_k} t} \cdot\frac{L_{i_k}t^2}{2} = \frac{(1-\beta)t}{2}.
	\end{eqnarray*}
	It implies that
	\begin{eqnarray*}
		f(z^{k+1}) &\le& f(z^k) - \frac{|\nabla_{i_k} f(z^k)|^2}{2L_{i_k}} + \frac{L_{i_k}}{2(1-\beta)^2}\cdot\frac{(1-\beta)^2t^2}{4}\\
		&=& f(z^k) - \frac{|\nabla_{i_k} f(z^k)|^2}{2L_{i_k}}  + \frac{L_{i_k} t^2}{8}
	\end{eqnarray*}
	and after taking expectation $\EE\left[\cdot\mid z^k\right]$ conditioned on $z^k$ from the both sides of the obtained inequality we get
	\begin{equation*}
		\EE\left[f(z^{k+1})\mid z^k\right] \le f(z^k) - \frac{1}{2}\EE\left[\frac{|\nabla_{i_k} f(z^k)|^2}{L_{i_k}}\mid z^k\right] + \frac{t^2}{8}\EE\left[L_{i_k}\mid z^k\right].
	\end{equation*}
	Note that 
	\begin{eqnarray*}
		\EE\left[\frac{|\nabla_{i_k} f(z^k)|^2}{L_{i_k}}\mid z^k\right] &=&\sum\limits_{i=1}^d\frac{p_i}{L_i}|\nabla_i f(z^k)|^2\\
		&\ge& \|\nabla f(z^k)\|_2^2\min\limits_{i=1,\ldots,d}\frac{p_i}{L_i}\\
		&\overset{\eqref{eq_sup:consequence_of_strong_convexity}}{\ge}& 2\mu\left( f(z^k)-f(x^*)\right)\min\limits_{i=1,\ldots,d}\frac{p_i}{L_i},
	\end{eqnarray*}
	since $\|\cdot\|_{\cD} = \|\cdot\|_2$, and
	\begin{eqnarray*}
		\EE\left[L_{i_k}\mid z^k\right] = \sum\limits_{i=1}^dp_iL_i.
	\end{eqnarray*}
	Putting all together we get
	\begin{eqnarray*}
		\EE\left[f(z^{k+1})\mid z^k\right] \le f(z^k) - \mu\min\limits_{i=1,\ldots,d}\frac{p_i}{L_i}\left(f(z^k) - f(x^*)\right) + \frac{t^2}{8}\sum\limits_{i=1}^dp_iL_i.
	\end{eqnarray*}
	Taking full expectation from the previous inequality we get
	\begin{equation*}
		\EE\left[f(z^{k+1}) - f(x^*)\right] \le \left(1 - \mu\min\limits_{i=1,\ldots,d}\frac{p_i}{L_i}\right)\EE\left[f(z^k)-f(x^*)\right] + \frac{t^2}{8}\sum\limits_{i=1}^dp_iL_i.
	\end{equation*}
	Since $\mu \le L_i$ for all $i=1,\ldots,d$ we have
	\begin{eqnarray*}
		\EE\left[f(z^k)  - f(x^*)\right] &\le& \left(1 - \mu\min\limits_{i=1,\ldots,d}\frac{p_i}{L_i}\right)^k\left(f(x^0) - f(x^*)\right)\\
		&&\quad + \left(\frac{t^2}{8}\sum\limits_{i=1}^dp_iL_i\right)\sum\limits_{l=0}^{k-1} \left(1 - \mu\min\limits_{i=1,\ldots,d}\frac{p_i}{L_i}\right)^l\\
		&\le& \left(1 - \mu\min\limits_{i=1,\ldots,d}\frac{p_i}{L_i}\right)^k\left(f(x^0) - f(x^*)\right)\\
		&&\quad + \left(\frac{t^2}{8}\sum\limits_{i=1}^dp_iL_i\right)\sum\limits_{l=0}^{\infty} \left(1 - \mu\min\limits_{i=1,\ldots,d}\frac{p_i}{L_i}\right)^l\\
		&=& \left(1 - \mu\min\limits_{i=1,\ldots,d}\frac{p_i}{L_i}\right)^k\left(f(x^0) - f(x^*)\right) + \frac{t^2}{8\mu\min\limits_{i=1,\ldots,d}\frac{p_i}{L_i}}\sum\limits_{i=1}^dp_iL_i.
	\end{eqnarray*}
	Lastly, from \eqref{eq_sup:str_cvx_sol_free_stepsizes_is} we have
	\begin{eqnarray*}
		\EE\left[f(z^K)\right] - f(x^*) &\le& \left(1 - \mu\min\limits_{i=1,\ldots,d}\frac{p_i}{L_i}\right)^{K}\left(f(x^0) - f(x^*)\right) + \frac{t^2}{8\mu\min\limits_{i=1,\ldots,d}\frac{p_i}{L_i}}\sum\limits_{i=1}^dp_iL_i\\
		&\overset{\eqref{eq_sup:str_cvx_sol_free_stepsizes_t_bound_is}}{\le}& \exp\left\{-K\mu\min\limits_{i=1,\ldots,d}\frac{p_i}{L_i}\right\}\left(f(x^0) - f(x^*)\right) + \frac{\varepsilon}{2}\\
		&\overset{\eqref{eq_sup:str_cvx_sol_free_stepsizes_number_of_iterations_is}}{\le}& \frac{\varepsilon}{2} + \frac{\varepsilon}{2} = \varepsilon.
	\end{eqnarray*}
\end{proof}

\subsection{Comparison of {\tt SMTP} and {\tt SMTP{\_}IS}}\label{sec_sup:comparison_with_smtp}
Here we compare {\tt SMTP} when $\cD$ is normal distribution with zero mean and $\frac{I}{d}$ covariance matrix with {\tt SMTP{\_}IS} with probabilities $p_i = \nicefrac{L_i}{\sum_{i=1}^dL_i}$. We choose such a distribution for {\tt SMTP} since it shows the best dimension dependence among other distributions considered in Lemma~\ref{lem_sup:aux_lemma}. Note that if $f$ satisfies Assumption~\ref{as_sup:coord_wise_L_smoothness}, it is $L$-smooth with $L = \max\limits_{i=1,\ldots,d}L_i$. So, we always have that $\sum_{i=1}^dL_i \le dL$. Table~\ref{tab:comparison} summarizes complexities in this case.

\begin{table*}[t]
\centering
\small
\renewcommand{\arraystretch}{1.25}{
\begin{tabular}{c|c|c|c|c|c}
\toprule
Assumptions on $f$     & \begin{tabular}{@{}c@{}}{\tt SMTP}\\ Compleixty \end{tabular} & Theorem & \begin{tabular}{@{}c@{}}Importance\\ Sampling \end{tabular} & \begin{tabular}{@{}c@{}}{\tt SMTP{\_}IS}\\Complexity \end{tabular} & Theorem  \\ 
  \midrule
None &  $\frac{\pi r_0\textcolor{red}{dL}}{\epsilon^2}$ & \ref{thm:non_cvx} &$p_i = \frac{L_i}{\sum_{i=1}^d L_i}$ & $\frac{2 r_0 \textcolor{red}{ d\sum_{i=1}^d L_i}}{\epsilon^2}$ 
&   \ref{thm:non_cvx_is}\\

Convex, $R_0 < \infty$  & $\frac{\pi R_{0,\ell_2}^2 \textcolor{red}{dL}}{2\varepsilon}\ln\left(\frac{2r_0}{\varepsilon}\right)$  & \ref{thm:cvx_constant_stepsize} & $p_i = \frac{L_i}{\sum_{i=1}^d L_i}$ & $\frac{R_{0,\ell_\infty}^2 \textcolor{red} { d\sum_{i=1}^d L_i}}{\epsilon} \ln\left(\frac{2r_0}{\varepsilon}\right)$  &   \ref{thm:cvx_constant_stepsize_is}\\

$\mu$-strongly convex & $ \frac{\pi\textcolor{red}{dL}}{2\mu}\ln\left(\frac{2r_0}{\varepsilon}\right)$ & \ref{thm:str_cvx_sol_free_stepsizes} & $p_i = \frac{L_i}{\sum_{i=1}^d L_i}$ &  $\frac{\textcolor{red}{\sum_{i=1}^d L_i}}{\mu}\ln\left(\frac{2r_0}{\varepsilon}\right) $& \ref{thm:str_cvx_sol_free_stepsizes_is} \\
\bottomrule
\end{tabular}
}
\vspace{-3pt}
\caption{Comparison of {\tt SMTP} with $\cD = N\left(0,\frac{I}{d}\right)$ and {\tt SMTP{\_}IS} with $p_i = \nicefrac{L_i}{\sum_{i=1}^dL_i}$. Here $r_0 = f(x^0)-f(x^*)$, $R_{0,\ell_2}$ corresponds to the $R_0$ from Assumption~\ref{ass_sup:bounded_level_set} with $\|\cdot\|_{\cD} = \|\cdot\|_2$ and $R_{0,\ell_\infty}$ corresponds to the $R_0$ from Assumption~\ref{ass_sup:bounded_level_set} with $\|\cdot\|_{\cD} = \|\cdot\|_1$.}
\label{tab:comparison}
\end{table*}

We notice that for {\tt SMTP} we have $\|\cdot\|_{\cD} = \|\cdot\|_2$. That is why one needs to compare {\tt SMTP} with {\tt SMTP{\_}IS} accurately. At the first glance, Table~\ref{tab:comparison} says that for non-convex and convex cases we get an extra $d$ factor in the complexity of {\tt SMTP{\_}IS} when $L_1 = \ldots =L_d = L$. However, it is natural since we use different norms for {\tt SMTP} and {\tt SMTP{\_}IS}. In the non-convex case for {\tt SMTP} we give number of iterations in order to guarantee $\EE\left[\|\nabla f(\overline{z}^K)\|_2\right] \le \varepsilon$ while for {\tt SMTP{\_}IS} we provide number of iterations in order to guarantee $\EE\left[\|\nabla f(\overline{z}^K)\|_1\right] \le \varepsilon$. From {\bluetext Holder's} inequality $\|\cdot\|_1 \le \sqrt{d}\|\cdot\|_2$ and, therefore, in order to have $\EE\left[\|\nabla f(\overline{z}^K)\|_1\right] \le \varepsilon$ for {\tt SMTP} we need to ensure that $\EE\left[\|\nabla f(\overline{z}^K)\|_2\right] \le \frac{\varepsilon}{\sqrt{d}}$. That is, to guarantee $\EE\left[\|\nabla f(\overline{z}^K)\|_1\right] \le \varepsilon$ {\tt SMTP} for aforementioned distribution needs to perform $\frac{\pi r_0{d^2L}}{\epsilon^2}$ iterations.

Analogously, in the convex case using Cauchy-Schwartz inequality $\|\cdot\|_2 \le \sqrt{d}\|\cdot\|_\infty$ we have that $R_{0,\ell_2} \le \sqrt{d}R_{0,\ell_\infty}$. Typically this inequality is tight and if we assume that $R_{0,\ell_\infty} \ge C\frac{R_{0,\ell_2}}{\sqrt{d}}$, we will get that {\tt SMTP{\_}IS} complexity is $\frac{R_{0,\ell_2}^2 { \sum_{i=1}^d L_i}}{\epsilon} \ln\left(\frac{2r_0}{\varepsilon}\right)$ up to constant factor.

That is, in all cases {\tt SMTP{\_}IS} shows better complexity than {\tt SMTP} up to some constant factor.

\section{Auxiliary results}\label{sec_sup:aux_res}

\begin{lem}[Lemma~3.4 from \cite{Bergou_2018}]\label{lem_sup:aux_lemma}
	Let $g\in \R^d$.  
\begin{enumerate}
\item If ${\cal D}$ is the uniform distribution on the unit sphere in $\R^d$, then
\begin{equation}
\gamma_{\cal D} = 1 \quad  \text{and} \quad \Exp_{s \sim {\cal D}}\;  | \<g, s> | \sim \frac{1}{\sqrt{2\pi d}} \|g\|_2.
\end{equation}
Hence, ${\cal D}$ satisfies Assumption~\ref{ass:stp_general_assumption}  with $\gamma_{\cal D}=1$, $\|\cdot\|_{\cal D}=\|\cdot\|_2$ and $\mu_{\cal D} \sim \frac{1}{\sqrt{2\pi d}}$.

\item 

If ${\cal D}$ is the normal distribution with zero mean and identity over ${d}$ as covariance matrix (i.e.\ $s\sim N(0,\frac{I}{d})$)
then
\begin{equation}
\gamma_{\cal D} = 1 \quad \text{and} \quad \Exp_{s \sim {\cal D}}\;  | \<g, s> | = \frac{\sqrt{2}}{  \sqrt{d\pi}}\|g\|_2.
\end{equation}
Hence, ${\cal D}$ satisfies Assumption~\ref{ass:stp_general_assumption}  with $\gamma_{\cal D}=1$, $\|\cdot\|_{\cal D}=\|\cdot\|_2$ and $\mu_{\cal D}= \frac{\sqrt{2}}{  \sqrt{d\pi}}$.

\item If ${\cal D}$ is the uniform distribution on  $\{e_1,\dots,e_d\}$, then
\begin{equation}
\gamma_{\cal D} = 1 \quad \text{and} \quad \Exp_{s\sim {\cal D}}\;  | \<g, s> |  = \frac{1}{d} \|g\|_1.
\end{equation}
Hence, ${\cal D}$ satisfies Assumption~\ref{ass:stp_general_assumption}  with $\gamma_{\cal D}=1$, $\|\cdot\|_{\cal D}=\|\cdot\|_1$ and $\mu_{\cal D}=\tfrac{1}{d}$.
\item If ${\cal D}$ is an arbitrary distribution on  $\{e_1,\dots,e_d\}$ given by $\PP\left\{s=e_i\right\}=p_i  >0$, then
\begin{equation}
\gamma_{\cal D} = 1 \quad \text{and} \quad \Exp_{s\sim {\cal D}}\;  | \<g, s> |  = \|g\|_{\cal D} \eqdef \sum_{i=1}^{\bluetext d} p_i |g_i|.
\end{equation}
Hence, ${\cal D}$ satisfies Assumption~\ref{ass:stp_general_assumption}  with $\gamma_{\cal D}=1$  
 and $\mu_{\cal D}=1$.
\item If ${\cal D}$ is a distribution on $D= \{u_1,\ldots,u_d\}$ where $u_1,\ldots,u_d$ form an orthonormal basis of $\R^d$ and 
$\PP\left\{s = d_i\right\} = p_i$, then 
\begin{equation}
\gamma_{\cal D} = 1 \quad \text{and} \quad \Exp_{s\sim {\cal D}}\;  | \<g, s> |  = \|g\|_{\cal D} \eqdef \sum_{i=1}^d p_i |g_i|.
\end{equation}
Hence, ${\cal D}$ satisfies Assumption~\ref{ass:stp_general_assumption}  with $\gamma_{\cal D}=1$  
 and $\mu_{\cal D}=1$.

\end{enumerate}	 
\end{lem}

\end{document}